\documentclass[12pt,a4paper]{report}
\usepackage[T2A]{fontenc}

\makeindex

\usepackage{inputenc}
\usepackage[russian]{babel}
\usepackage{amsmath}
\usepackage{amssymb}
\usepackage{amsthm}
\usepackage{latexsym}
\usepackage{comment}        		
\usepackage{colortbl} 

\renewcommand{\phi}{\varphi}
\renewcommand{\epsilon}{\varepsilon}
\renewcommand{\kappa}{\varkappa}       
\renewcommand{\ge}{\geqslant}           

\def\PI{{\rm PI}}
\def \GK{\operatorname{GK}}
\def\Pid{\operatorname{Pid}}
\def\Var{\operatorname{Var}}
\def\Ht{\operatorname{Ht}}
\def\AC{\operatorname{AC}}

\def\Mat{\operatorname{Mat}}

\newcommand{\cM}{{\cal M}}

\newcommand{\gA}{{\mathfrak A}}
\newcommand{\gB}{{\mathfrak B}}

\newcommand{\FF}{\mathbb{F}}

\newtheorem{theorem}{Теорема}[section]
\newtheorem{corollary}{Следствие}[section]
\newtheorem{lemma}{Лемма}[section]
\newtheorem{proposition}{Предложение}[section]
\newtheorem{notation}{Обозначение}[section]
\newtheorem{remark}{Замечание}[section]
\newtheorem{construction}{Конструкция}[section]
\newtheorem{algorithm}{Алгоритм}[section]
\newtheorem{ques}{Вопрос}[section]
\newtheorem{question}{Вопрос}[section]
\newtheorem{problem}{Проблема}[section]

\theoremstyle{definition}
\newtheorem{definition}{Определение}[section]

\newtheorem{example}{\rm\bf Пример}[section]

\begin{document}

\renewcommand{\baselinestretch}{1.2} 

{\large

\rightline{УДК 512.5+512.64+519.1}

\vskip 6mm

\begin{center}
{ \large Оценки, связанные с теоремой Ширшова о высоте}

\end{center}

\vskip 11mm
\begin{center}
Харитонов Михаил Игоревич\footnote{Работа выполнена при частичной финансовой поддержке фонда Дмитрия Зимина ``Династия'' и фонда О.В. Дерипаски; исследование поддержано грантом РФФИ \No 14-01-00548.}
\end{center}

\vskip 15mm
\begin{flushright}
Работа посвящена 80-летию\\ профессора В.~Н.~Латышева
\end{flushright}

}

\tableofcontents

\newpage

\chapter{Проблемы Бернсайдовского типа и тождества в теории колец}

\section{Теория колец в контексте проблематики\\ Бернсайдовского типа}

Дальнейшее развитие теории периодических групп требует решения следующей проблемы, поставленной Бернсайдом:

\begin{problem}[\cite{Bur1902}]\index{Проблема!Бернсайда!для групп}
Будет ли конечной всякая периодическая конечно порожд\"{е}нная группа?
\end{problem}

Первоначальные усилия были направлены в сторону положительного решения проблемы, так как все известные частные случаи давали позитивный ответ. Например, если группа порождена $m$ элементами и порядок каждого е\"{е} элемента является делителем числа 4, она конечна. Первый контрпример к
``неограниченной'' проблеме был найден Е.~С.~Голодом в $1964$ году на основе
универсальной конструкции Голода--Шафаревича. Вопрос о
локальной конечности групп с тождеством $x^n=1$ был решен отрицательно
в знаменитых работах П.~С.~Новикова и С.~И.~Адяна $(1968)$: было доказано существование
для любого нечетного $n\ge 4381$ бесконечной группы с $m>1$ образующими, удовлетворяющей
тождеству $x^n=1$. Эта оценка была улучшена до $n\ge 665$ С.~И.~Адяном $(1975)$.
Позднее А.~Ю.~Ольшанский предложил геометрически наглядный вариант доказательства
для нечетных $n>10^{10}$.

Построения в полугруппах, как правило, проще, чем в группах. Например, вопрос
о существовании конечно порожд\"{е}нной ниль-полугруппы, то есть полугруппы, каждый
элемент которой в некоторой степени обращается в нуль, имеет тривиальный положительный
ответ: уже в алфавите из двух букв имеются слова сколь угодно большой длины, не
содержащие трех подряд одинаковых подслов (кубов). Этот факт был независимо доказан А.~Туэ (\cite{Th1906}) и М.~Морсом (\cite{Mor21}).

\begin{theorem}[Морс--Туэ] \index{Теорема!Морса--Туэ}
Пусть $X = \{a, b\}$, $X^*$ --- множество слов над алфавитом $X$, подстановка $\phi$ задана соотношениями $\phi(a)=ab$, $\phi(b)=ba$. Тогда если слово $w\in X^*$ --- бескубное, то и $\phi(w)$ --- бескубное.
\end{theorem}
В дальнейшем этот результат был усилен Туэ:
\begin{theorem}[Туэ--1, \cite{Th1906}]\index{Теорема!Туэ}
 Пусть $X = \{a, b, c\}$, $X^*$ --- множество слов над алфавитом $X$, подстановка $\phi$ задана соотношениями $\phi(a)=abcab$, $\phi(b)=acabcb$, $\phi(c)=acbcacb$. Тогда если слово $w\in X^*$ --- бесквадратное, то и $\phi(w)$ --- бесквадратное.
\end{theorem}
\begin{theorem}[Туэ--2]\index{Теорема!Туэ}
 Пусть $L$ и $N$ --- алфавиты, $N^*$ --- множество слов над алфавитом $N$, для подстановки $\phi:L\rightarrow N^*$ выполнены следующие условия:
\begin{enumerate}
\item если длина $w$ не больше 3, то $\phi(w)$ --- бесквадратное;
\item если $a$, $b$ --- буквы алфавита $L$, а $\phi(a)$ --- подслово $\phi(b)$, то $a=b$. 
\end{enumerate}
Тогда если слово $w\in L^*$ --- бесквадратное, то и $\phi(w)$ --- бесквадратное.
\end{theorem}
Полное алгоритмическое описание бесквадратных подстановок было впервые получено Дж.~Берстелем (\cite{Ber77_1, Ber77_2}). В дальнейшем Крошмором было предложено следующее описание:
\begin{theorem}[Крошмор, \cite{Cr82}] \index{Теорема!Крошмора}
 Пусть $L$ и $N$ --- алфавиты, $N^*$ --- множество слов над алфавитом $N$, $\phi:L\rightarrow N^*$ --- подстановка, $M$ --- наибольший размер образа буквы алфавита $L$ при подстановке $\phi$, $m$ --- наименьший размер образа буквы $L$ при той же подстановке, $k=max\{3, 1+[(M-3)/m]\}$. Тогда подстановка $\phi$ --- бесквадратная в том и только в том случае, когда для любого бесквадратного слова $w$ длины $\leqslant k$ слово $\phi(w)$ будет бесквадратным.
\end{theorem}

Обзор проблем Бернсайдовского типа в теории колец провед\"{е}н в монографии М.~Сапира (\cite{Sap14}) и статье А.~Я.~Белова (\cite{Bel07}).

Проблема Бернсайда для ассоциативных алгебр была  сформулирована А.~Г.~Курошем в тридцатых годах двадцатого века:

\begin{question}\index{Проблема!Бернсайда!для ассоциативных алгебр}
Пусть все 1-порожд\"{е}нные подалгебры конечно порожд\"{е}нной ассоциативной алгебры $A$ конечномерны. Будет ли $A$ конечномерна?
\end{question}

\begin{proposition}
Пусть $A$ --- ассоциативная $K$-алгебра, $K$ --- коммутативное кольцо, $a\in A$.  Подалгебра, порожд\"{е}нная $a$, конечномерна тогда и только тогда, когда $a$ --- алгебраический элемент.
\end{proposition}

Отрицательный ответ был получен Е.~С.~Голодом в 1964 году более сложным способом, чем  в случае полугрупп. Например, в полугрупповом случае найд\"{е}тся 3-порожд\"{е}нная бесконечная полугруппа, удовлетворяющая тождеству $x^2=0$. Для ассоциативных алгебр над полем характеристики $\geqslant 3$ это невозможно. 

\begin{definition}\index{Класс нильпотентности} \index{Ниль-индекс} \index{Индекс!нильпотентности}
{\em Классом нильпотентности}, {\em индексом нильпотентности} или {\em ниль-индексом} ассоциативной алгебры $A$ называется минимальное натуральное число $n$ такое, что $A^n=0$.
\end{definition}

\begin{theorem}[Курош, \cite{Kur41}]
Любая удовлетворяющая тождеству $x^2=0$ алгебра над полем характеристики $\geqslant 3$ или 0 является нильпотентной класса 3. Любая нильпотентная конечно порожд\"{е}нная алгебра конечномерна.
\end{theorem}

\begin{definition} \index{Индекс!алгебры}
{\em Индексом} алгебраической алгебры $A$ называется супремум степеней минимальных аннулирующих многочленов элементов $A$.
\end{definition}

В 1941 году Курош в работе \cite{Kur41} сформулировал проблему Бернсайда  алгебр конечного индекса:

\begin{question}\index{Проблема!Куроша} 
\begin{enumerate}
	\item Верно ли, что конечно порожд\"{е}нная ниль-алгебра конечного ниль-индекса нильпотента? 
	\item Верно ли, что конечно порожд\"{е}нная алгебра конечного индекса конечномерна?
\end{enumerate}
\end{question}

В 1946 году  И.~Капланский \cite{Kap46} и Д.~Левицкий \cite{Lev46} ответили на эти вопросы положительным образом для алгебр с допустимым полиномиальным тождеством, где полиномиальное тождество называется {\em допустимым}, если один из его коэффициентов не равен 1. Заметим, что в случае ассоциативных алгебр над полями любое полиномиальное тождество является допустимым.

В 1948 году И.~Капланский отказался от условия конечности индекса:

\begin{theorem}[Капланский, \cite{Kap48}]\label{essay:kap} \index{Теорема!Капланского}
Любая конечно порожд\"{е}нная алгебраическая алгебра над коммутаитвным кольцом, удовлетворяющая допустимому полиномиальному тождеству, конечномерна.
\end{theorem}

 Доказательства в работах \cite{Kap46} и \cite{Lev46} были проведены структурными методами.  Заметим, что структурная теория, развитая в работах Ш.~Амицура, И.~Капланского и др.,
позволила решить ряд классических проблем и служит основой для дальнейших
исследований. Обычная схема структурных рассуждений состояла в исследовании полупростой части (матриц над телами) и редукции к полупростой ситуации
пут\"{е}м факторизации по радикалу. Несмотря на свою эффективность, рассуждения такого рода не являются
конструктивными. Кроме того, доказательства, которые получаются с помощью
структурной теории, не дают понимания происходящего ``на микроуровне'', т. е.
на уровне слов и соотношений между ними.

В 1958 году А.~И.~Ширшов доказал свою знаменитую теорему о высоте чисто комбинаторными методами \cite{Sh53, Sh54, Sh57_1, Sh57_2, Sh58, Sh62(1), Sh62(2), BBL97}. Из теоремы Ширшова о высоте следует решение проблемы Куроша для ассоциативных $\PI$-алгебр в усиленной форме, т.е. требование алгебраичности алгебры заменяется на требование алгебраичности слов от порождающих длины менее степени тождества в алгебре, а требование конечности отбрасывается. Таким образом, результат А.~И.~Ширшова есть улучшение теоремы \ref{essay:kap}.

\begin{theorem}[Ширшов, \cite{Sh57_1, Sh57_2}]\label{essay:shirshov}\index{Теорема!Ширшова о высоте}
Пусть $A=\langle X\rangle$ --- конечно порожд\"{е}нная ассоциативная алгебра над коммутативным кольцом, удовлетворяющая допустимому полиномиальному тождеству степени $n$. Тогда найд\"{е}тся число $H$, зависящее только от $|X|$ и $n$ такое, что каждый элемент $a\in A$ может быть представлен как линейная комбинация слов вида $v_1^{n_1}\cdots v_h^{n_h}$, где $h\leqslant H$, а длина каждого слова $v_i$ меньше $n$.
\end{theorem}

Если в ассоциативной алгебре есть допустимое полиномиальное тождество, то есть и допустимое полилинейное тождество той же или меньшей степени. Доказательство этого факта можно найти в обзоре \cite{SGS78}. 

Пусть $f=0$ --- допустимое полилинейное тождество степени $n$. Тогда каждый моном $f$ является произведением переменных в некотором порядке (каждая переменная встречается в точности один раз в каждом мономе). Таким образом, все мономы $f$ получаются из $x_1 x_2\cdots x_n$ пут\"{е}м перестановки переменных. Следовательно, каждое допустимое полилинейное тождество имеет форму 

$$\sum\limits_{\sigma\in S_n}\alpha_\sigma x_{\sigma(1)} x_{\sigma(2)} x_{\sigma(n)},$$

где $S_n$ --- группа всех перестановок множества $\{1,\dots , n\}$, а один из коэффициентов $\alpha_\sigma$ равен 1. 

Это тождество после переименования переменных может быть представлено в форме 

$$x_1 x_2\dots x_n = \sum\limits_{\sigma\in S_n\setminus\{1\}}\beta_\sigma x_{\sigma(1)} x_{\sigma(2)} x_{\sigma(n)}$$

Таким образом, каждое произведение $u_1 u_2\cdots u_n$ элементов алгебры $A$ есть линейная комбинация перестановок этого произведения. Пусть на словах из $X^*$ задан некоторый частичный порядок. Рассмотрим подмножесво, получающихся из $X^*$ выбрасыванием слов, представимых в виде линейной комбинации меньших слов. Получаем, что теорему \ref{essay:shirshov} можно доказывать только для этого подмножества. Таким образом доказательство теоремы \ref{essay:shirshov} сводится к чистой комбинаторике.

\medskip
Заметим, что верно и утверждение, обратное теоремам \ref{essay:kap} и \ref{essay:shirshov}.

Известно, что в каждой ассоциативной алгебре размерности $n$ над коммутативным кольцом выполняется так называемое {\em стандартное тождество} \index{Тождество!стандартное} $S_{n+1}(x_1,\dots, x_{n+1})=0$, где $S_n(x_1,\dots, x_{n}) = \sum\limits_{\sigma\in S_n}(-1)^\sigma  x_{\sigma(1)} x_{\sigma(2)} x_{\sigma(n)}$, $S_n$ --- группа перестановок. 

Отсюда получаем следующее предложение:

\begin{proposition}
Любая конечномерная алгебра является конечно порожд\"{е}нной, алгебраичной и удовлетворяет допустимому полилинейному тождеству.
\end{proposition}

\begin{definition} \index{Размерность Гельфанда--Кириллова}
$\GK(A)$ --- {\em размерность Гельфанда--Кириллова алгебры $A$} --- определяется по правилу
$$\GK(A)=\lim_{n\to\infty}{\ln V_A(n)\over\ln(n)},$$
где $V_A(n)$ есть\index{Функция роста алгебры} {\it функция роста алгебры $A$}, т.е.
размерность векторного пространства, порожд\"{е}нного словами степени
не выше $n$ от образующих $A$.
\end{definition}

\begin{corollary}[Berele]
Пусть $A$~--- конечно порожд\"{е}нная $\PI$-алгебра.\linebreak
Тогда $\GK(A)<\infty.$
\end{corollary}


\begin{notation}
Обозначим как $\deg(A)$ {\em степень алгебры}, т.е.
минимальную степень тождества, которое в ней выполняется.
Через $\Pid(A)$ обозначим {\em сложность} алгебры $A$, т.е. максимальное $k$
такое, что ${\Bbb M}_k$ --- алгебра матриц размера $k$ --- принадлежит
многообразию $\Var(A)$, порожд\"{е}нному алгеброй $A$.
\end{notation}

Введ\"{е}м понятие высоты, частный случай которого использовался в теореме \ref{essay:shirshov}.

\begin{definition} \index{Высота!множества}
Назов\"{е}м множество $\cM\subset X^*$ множеством  {\em ограниченной высоты
$h=\Ht_Y(A)$ над множеством слов $Y= \{ u_1, u_2,\ldots\}$}, если $h$ --- минимальное число такое, что любое слово $u\in \cM$ либо $n$-разбиваемо, либо представимо в виде $u =  u_{j_1}^{k_1} u_{j_2}^{k_2}\cdots u_{j_r}^{k_r}$, где $r\leqslant h$.
\end{definition}

\begin{definition} \index{Высота!алгебры} 
Назов\"{е}м \PI-алгебру $A$ алгеброй {\em ограниченной высоты
$h=\Ht_Y(A)$ над множеством слов $Y = \{ u_1, u_2,\ldots\}$}, если
$h$ --- минимальное число такое, что любое слово $x$ из $A$ можно
представить в виде
$$x = \sum_i \alpha_i u_{j_{(i,1)}}^{k_{(i,1)}}
u_{j_{(i, 2)}}^{k_{(i,2)}}\cdots
u_{j_{(i,r_i)}}^{k_{(i,r_i)}},
$$
причем $\{r_i\}$ не превосходят $h$. Множество \index{Базис!Ширшова}
$Y$ называется {\em базисом Ширшова} или {\em $s$-базисом} для алгебры $A$.
\end{definition}

Вместо понятия {\it высоты} иногда удобнее пользоваться близким понятием
{\it существенной высоты}.

\begin{definition} \index{Высота!существенная}
Алгебра $A$ имеет {\em существенную высоту $h=H_{Ess}(A)$} над
конечным множеством $Y$, называемым {\em $s$-базисом алгебры $A$}, если можно
выбрать такое конечное множество $D\subset A$, что $A$ линейно
представима элементами вида $t_1\cdot\ldots\cdot t_l$, где $l\leqslant
2h+1$, и $\forall i (t_i\!\in\! D \vee t_i=y_i^{k_i};y_i\in Y)$,
причем множество таких $i$, что $t_i\not\in D$, содержит не более
$h$ элементов. Аналогично определяется {\em существенная высота}
множества слов.
\end{definition}

Говоря неформально, любое длинное слово есть произведение
периодических частей и ``прокладок'' ограниченной длины.
Существенная высота есть число таких периодических кусков, а
обычная еще учитывает ``прокладки''.

В связи с теоремой о высоте возникли следующие вопросы:

\begin{enumerate}

\item На какие классы колец можно распространить теорему о высоте?

\item Над какими $Y$ алгебра $A$ имеет ограниченную высоту? В
частности, какие наборы слов можно взять в качестве $\{v_i\}$?

\item Как устроен вектор степеней $(k_1,\ldots,k_h)$? Прежде
всего: какие множества компонент этого вектора являются
существенными, т.е. какие наборы $k_i$ могут быть одновременно
неограниченными? Какова существенная высота? Верно ли, что
множество векторов степеней обладает теми или иными свойствами
регулярности?

\item Как оценить высоту?
\end{enumerate}

Перейдем к обсуждению поставленных вопросов.

\section{Неассоциативные обобщения}

 Теорема о высоте была
распространена на некоторые классы колец, близких к ассоциативным.
С.~В.~Пчелинцев \cite{Pchelintcev} доказал ее для альтернативного
и $(-1,1)$ случаев, С.~П.~Мищенко \cite{Mishenko1} получил аналог
теоремы о высоте для алгебр Ли с разреженным тождеством. В работе
автора \cite{Belov1} теорема о высоте была доказана для некоторого
класса колец, асимптотически близких к ассоциативным, куда
входят, в частности, альтернативные и йордановы $\PI$-алгебры.

\section{Базисы Ширшова} \index{Базис!Ширшова}

\begin{theorem}  [А.~Я.~Белов]         \label{ThKurHmg}
а) Пусть $A$~--- градуированная $\PI$-алгебра, $Y$~--- конечное
множество однородных элементов. Тогда если при всех $n$ алгебра
$A/Y^{(n)}$ нильпотентна, то $Y$ есть $s$-базис $A$. Если при этом
$Y$ порождает $A$ как алгебру, то $Y$~--- базис Ширшова алгебры
$A$.

б) Пусть $A$~--- $\PI$-алгебра, $M\subseteq A$~--- некоторое курошево
подмножество в $A$. Тогда $M$~--- $s$-базис алгебры $A$.
\end{theorem}

$Y^{(n)}$ обозначает идеал, порожд\"{е}нный $n$-ыми степенями элементов
из $Y$. Множество $M\subset A$ называется {\it курошевым}, если
любая проекция $\pi\colon A\otimes K[X]\to A'$, в которой образ
$\pi(M)$ цел над $\pi(K[X])$, конечномерна над $\pi(K[X])$.
Мотивировкой этого понятия служит следующий пример.  Пусть
$A={\Bbb Q}[x,1/x]$. Любая проекция $\pi$ такая, что $\pi(x)$
алгебраичен, имеет конечномерный образ. Однако множество $\{x\}$
не является $s$-ба\-зис\-ом алгебры ${\Bbb Q}[x,1/x]$. Таким
образом, ограниченность существенной высоты есть некоммутативное
обобщение свойства {\it целости}.

Описание базисов\index{Базис!Ширшова!состоящий из слов}
Ширшова, состоящих из слов, заключено в следующей теореме:

\begin{theorem}[\cite{BBL97, BR05}]            \label{ThBelheight}
Множество слов $Y$ является базисом Ширшова алгебры $A$ тогда и
только тогда, когда для любого слова $u$ длины не выше $m =
\Pid(A)$~--- сложности алгебры $A$~--- множество $Y$ содержит
слово, циклически сопряженное к некоторой степени слова $u$.
\end{theorem}

Аналогичный результат был независимо получен Г.~П.~Чекану и В.~Дренски. Вопросы, связанные с локальной конечностью алгебр, с алгебраическими множествами слов степени не выше сложности алгебры, исследовались в работах \cite{Ufn90, Cio97, Cio88, CK93, Che94, Ufn85, UC85}. В этих же работах обсуждались вопросы, связанные с обобщением теоремы о независимости.

\section{Существенная высота} \index{Высота!существенная}
Ясно, что размерность
Гель\-фан\-да--Ки\-рил\-ло\-ва оценивается существенной высотой и
что $s$-базис является базисом Ширшова тогда и только тогда,
когда он порождает $A$ как алгебру. В представимом случае имеет
место и обратное утверждение.

\begin{theorem}[А.~Я.~Белов, \cite{BBL97}]
Пусть $A$~--- конечно порожд\"{е}нная представимая алгебра и пусть
$H_{Ess}{}_Y(A)<\infty$. Тогда $H_{Ess}{}_Y(A)=\GK(A)$.
\end{theorem}

\begin{corollary}[В.~Т.~Марков]
Размерность Гель\-фан\-да--Ки\-рил\-ло\-ва ко\-неч\-но
по\-рож\-ден\-ной представимой алгебры есть целое число.
\end{corollary}

\begin{corollary}
Если $H_{Ess}{}_Y(A)<\infty$ и алгебра $A$ представима, то
$H_{Ess}{}_Y(A)$ не зависит от выбора $s$-базиса $Y$.
\end{corollary}

В этом случае размерность Гель\-фан\-да--Ки\-рил\-ло\-ва также
равна существенной высоте в силу локальной представимости относительно свободных алгебр.

\section{Строение векторов степеней} Хотя в представимом случае
размерность Гель\-фан\-да--Ки\-рил\-ло\-ва и существенная высота
ведут себя хорошо, тем не менее даже тогда множество векторов
степеней может быть устроено плохо~--- а именно, может быть
дополнением к множеству решений системы
экс\-по\-нен\-ци\-аль\-но-по\-ли\-но\-ми\-аль\-ных диофантовых
уравнений \cite{BBL97}. Вот почему существует пример представимой
алгебры с трансцендентным рядом Гильберта. Однако для относительно
свободной алгебры ряд Гильберта рационален \cite{Belov501}.

\section{$n$-разбиваемость, обструкции и теорема Дилуорса} \index{Слово!$n$-разбиваемое}\index{Теорема!Дилуорса}

Значение понятия {\it
$n$-раз\-би\-ва\-ем\-ос\-ти} выходит за рамки проблематики,
относящейся к проблемам бернсайдовского типа. Оно играет роль и
при изучении полилинейных слов, в оценке их количества, где {\it
полилинейным} называется слово, в которое каждая буква входит
не более одного раз. В.~Н.~Латышев применил теорему Дилуорса для
получения оценки числа не являющихся $m$-разбиваемыми полилинейных слов степени $n$ над
алфавитом $\{a_1,\dots,a_n\}$ (см. \cite{Lat72}). Эта
оценка:  ${(m - 1)}^{2n}$ и она близка к реальности. Напомним эту
теорему.

\begin{theorem}[Дилуорс, \cite{Dil50}]
Пусть $n$ --- наибольшее количество
элементов антицепи данного конечного частично упорядоченного
множества $M$. Тогда $M$ можно разбить на $n$  попарно
непересекающихся цепей.
\end{theorem}
Рассмотрим полилинейное слово $W$ из $n$ букв. \index{Слово!полилинейное}
Положим $a_i\succ
a_j$, если $i>j$ и буква $a_i$ стоит в слове $W$ правее $a_j$.
Условие не $m$-разбиваемости означает отсутствие антицепи из $m$
элементов. Тогда по теореме Дилуорса все позиции (и,
соответственно, буквы $a_i$) разбиваются на $(m-1)$ цепь. Сопоставим
каждой цепи свой цвет. Тогда раскраска позиций и раскраска букв
однозначно определяет слово $W$. А число таких раскрасок не
превосходит $(m-1)^n\times (m-1)^n=(m-1)^{2n}$. Улучшение этой оценки и другие вопросы, связанные с полилинейными словами, рассматриваются в главе \ref{ch:pu_number}.

В.~Н.~Латышев (\cite{Lat72}) с помощью привед\"{е}нных оценок пров\"{е}л 
прозрачное доказательство теоремы Регева:
\begin{theorem}[Регев, \cite{Reg71}] \index{Теорема!Регева}
Если алгебры $A$ и $B$ удовлетворяют полиномиальному тождеству, то алгебра  $A\otimes_F B$ также удовлетворяет полиномиальному тождеству. 
\end{theorem}

Вопросы, связанные с перечислением полилинейных слов, не
являющихся $n$-раз\-би\-ва\-е\-мы\-ми, имеют самостоятельный
интерес. (Например, существует биекция между не
$3$-раз\-би\-ва\-е\-мы\-ми словами и числами Каталана.) С одной
стороны, это чисто комбинаторная задача, с другой стороны, она
связана с рядом коразмерностей для алгебры общих матриц.
Исследование полилинейных слов представляется чрезвычайно важным.

В 1950 году Шпехт (см. \cite{Sp50}) поставил проблему существования бесконечно базируемого многообразия ассоциативных алгебр над полем характеристики 0. Решение проблемы Шпехта для нематричного случая представлено в докторской диссертации В. Н. Латышева \cite{Lat77}. Рассуждения В. Н. Латышева основывались на применении техники частично упорядоченных множеств. А. Р. Кемер (см. \cite{Kem87} доказал, что каждое многообразие ассоциативных алгебр конечно базируемо, тем самым решив проблему Шпехта.

Первые примеры бесконечно базируемых ассоциативных колец были получены А. Я. Беловым (\cite{Bel99}), А. В. Гришиным (\cite{Gr99}) и В. В. Щиголевым (\cite{Shch99}).

Введ\"{е}м теперь некоторый порядок на словах алгебры над полем. Назов\"{е}м {\it обструкцией} полилинейное слово, которое
\begin{itemize}
    \item является уменьшаемым (т. е. является комбинацией меньших слов);
    \item не имеет уменьшаемых подслов;
    \item не является изотонным образом уменьшаемого слова меньшей длины.
\end{itemize}

В.~Н.~Латышев  (\cite{LatyshevMulty}) поставил
проблему конечной базируемости множества старших полилинейных слов
для $T$-идеала относительно взятия надслов и изотонных
подстановок.

\begin{ques}[Латышев]
Верно ли, что количество обструкций для полилинейного $T$-идеала конечно?
\end{ques}

Из проблемы Латышева вытекает полилинейный случай проблемы конечной базируемости для алгебр над полем конечной характеристики. Наиболее важной обструкцией является обструкция $x_n x_{ n-1}\dots x_1$, е\"{е} изотонные образы составляют множество не $n$-разбиваемых слов. 

В связи с этими вопросами возникает проблема:

\begin{ques}
Перечислить количество полилинейных слов, отвечающих данному конечному набору обструкций. Доказать элементарность соответствующей производящей функции.
\end{ques}
Также имеется тесная связь с проблемой слабой
н\"{е}теровости групповой алгебры бесконечной финитарной
симметрической группы над полем положительной характеристики (для
нулевой характеристики это было установлено А.~Залесским). Для
решения проблемы Латышева надо уметь переводить свойства
$T$-идеалов на язык полилинейных слов. В работах
\cite{BBL97,Belov1} была предпринята попытка осуществить программу перевода
структурных свойств алгебр на язык комбинаторики слов. На язык
полилинейных слов такой перевод осуществить проще, в дальнейшем
можно получить информацию и о словах общего вида.

В работе техника В.~Н.~Латышева переносится на
неполилинейный случай, что позволяет получить субэкспоненциальную
оценку в теореме Ширшова о высоте. Г.~Р.~Челноков предложил идею этого переноса в 1996 году.

\section{Оценки высоты}

Первоначальное доказательство А.~И.~Ширшова
хотя и было чисто комбинаторным (оно основывалось на технике
элиминации, развитой им в алгебрах Ли, в частности, в
доказательстве теоремы о свободе), однако оно давало только
упрощ\"{е}нные рекурсивные оценки. Позднее А.~Т.~Колотов
\cite{Kolotov} получил оценку на $\Ht(A)\leqslant l^{l^n}$\
(Здесь и далее: $n=\deg(A)$,\, $l$~--- число образующих). А.~Я.~Белов в работе \cite{Bel92} показал, что $\Ht(n,l)<2nl^{n+1}$.  Экспоненциальная оценка теоремы Ширшова о высоте изложена также в работах \cite{BR05,Dr00,Kh11(2)}. Данные оценки улучшались в работах А.~Клейна \cite{Klein,Klein1}. В 2001 году Е.~С.~Чибриков в работе \cite{Ch01} доказал, что $\Ht(4,l) \geqslant  (7k^2-2k).$  Верхние и нижние оценки на структуру кусочной периодичности, полученные М.~И.~Харитоновым в работах \cite{Kh11, Kh11(2)}, изложены в главе \ref{ch:pp}. 

В работе \cite{LS13} рассматривается связь между высотой конечных и бесконечных полей одинаковой характеристики. Пусть $A$ --- ассоциативная алгебра надо конечным полем $\FF$ из $q$ элементов и тождеством степени $n$. Тогда доказано, что если $q\geqslant n$, то индекс нильпотентности алгебры будем таким же, что и в случае бесконечного поля. Если же $q=n-1$, то индексы нильпотентности в случае поля $\FF$ и бесконечного поля отличается не более, чем на~1.

В 2011 году А.~А.~Лопатин \cite{Lop11} получил следующий результат:

\begin{theorem}
Пусть $C_{n,l}$ --- степень нильпотентности свободной $l$-порожд\"{е}нной алгебры и удовлетворяющей тождеству $x^n=0.$ Пусть $p$ --- характеристика базового поля алгебры --- больше чем $n/2.$ Тогда
$$(1): C_{n,l}<4\cdot 2^{n/2} l.$$
\end{theorem}
По определению $C_{n,l}\leqslant \Psi(n, n, l).$

Заметим, что для малых $n$ оценка  (1) меньше, чем полученная в данной работе оценка $\Psi(n, n, l),$ но при росте $n$ оценка $\Psi(n,n,l)$ асимптотически лучше оценки (1).

Е.~И.~Зельманов поставил следующий вопрос в Днестровской тетради \cite{Dnestrovsk} в 1993 году:
\begin{ques}
Пусть $F_{2,m}$ --- свободное $2$-порожд\"{е}нное ассоциативное кольцо с тождеством $x^m=0.$ Верно ли, что класс нильпотентности кольца $F_{2,m}$ раст\"{е}т экспоненциально по $m?$
\end{ques}

В работе получен следующий ответ на вопрос Е.~И.~Зельманова: в действительности искомый класс нильпотентности раст\"{е}т субэкспоненциально.

\section{Основные результаты}

В работе получен ответ на вопрос \ref{ques} Е.~И.~Зельманова: в действительности искомый класс нильпотентности раст\"{е}т субэкспоненциально.

\begin{theorem}     \label{c:main2}
Высота множества не $n$-разбиваемых слов над $l$-буквенным
алфавитом относительно множества слов длины меньше $n$ не
превышает $\Phi(n,l)$, где
$$\Phi(n,l) = 2^{96} l\cdot n^{12\log_3 n + 36\log_3\log_3 n + 91}.$$
\end{theorem}
Из данной теоремы путем некоторого огрубления и упрощения оценки получается, что
при фиксированном $l$ и $n \rightarrow\infty$
$$\Phi(n,l) = n^{12(1+o(1))\log_3{n}},$$

а при фиксированном $n$ и $l\rightarrow\infty$
$$\Phi(n,l) < C(n)l.$$

Также доказательство этих результатов содержится в работе \cite{BK}.

\begin{corollary}
Высота $l$-порожд\"{е}нной $\PI$-алгебры с допустимым полиномиальным
тождеством степени $n$ над множеством слов длины меньше $n$ не
превышает $\Phi(n,l)$.
\end{corollary}



Как следствие получаются субэкспоненциальные оценки на индекс
нильпотентности $l$-порожд\"{е}нных ниль-алгебр степени $n$ для
произвольной характеристики.

Другим основным результатом диссертации является следующая теорема:

\begin{theorem}      \label{c:main1}
Пусть $l$, $n$ и $d\geqslant n$ --- некоторые натуральные числа. Тогда все
$l$-порожд\"{е}нные слова длины не меньше, чем $\Psi(n,d,l)$, либо
содержат $x^d$, либо являются $n$-разбиваемыми, где
$$
\Psi(n,d,l)=2^{27} l (nd)^{3 \log_3 (nd)+9\log_3\log_3 (nd)+36}.
$$
\end{theorem}
Из данной теоремы путем некоторого огрубления и упрощения оценки получается, что
при фиксированном $l$ и $nd \rightarrow\infty$
$$\Psi(n,d,l) = (nd)^{3(1+o(1))\log_3(nd)},$$

а при фиксированных $n, d$ и $l\rightarrow\infty$
$$\Psi(n,d,l) < C(n,d)l.$$

\begin{corollary}
 Пусть $l$, $d$ --- некоторые натуральные числа.
Пусть в ассоциативной  $l$-порожд\"{е}нной алгебре $A$ выполнено тождество
$x^{d}=0$. Тогда ее индекс нильпотентности меньше, чем
$\Psi(d,d,l)$.
\end{corollary}

Кроме того, доказывается субэкспоненциальная оценка, которая лучше при малых $n$ и $d$: 

\begin{theorem}     \label{t2:log_2} Пусть $l$, $n$ и $d\geqslant n$ --- некоторые натуральные числа. Тогда все $l$-порожд\"{е}нные слова длины не меньше, чем $\Psi(n,d,l)$, либо содержат $x^d$, либо являются $n$-разбиваемыми, где $$\Psi(n,d,l) = 256 l(nd)^{2\log_2 (nd)+10}d^2.$$ \end{theorem}

\begin{notation}
Для вещественного числа $x$ положим $\ulcorner x\urcorner := -[-x].$ Таким образом мы округляем нецелые числа в большую сторону.
\end{notation}
В процессе доказательства теоремы \ref{c:main2} доказывается следующая теорема, оценивающая существенную высоту:

\begin{theorem} \label{ThThick}
Существенная высота $l$-порожд\"{е}нной $PI$-алгебры с допустимым полиномиальным тождеством степени $n$ над множеством слов длины меньше $n$ меньше, чем $\Upsilon (n, l),$ где
$$\Upsilon (n, l) = 2n^{3\ulcorner\log_3 n\urcorner + 4} l.$$
\end{theorem}

\section{О нижних оценках}

Сравним полученные результаты  с нижней
оценкой для высоты. Высота алгебры $A$ не меньше ее размерности
Гель\-фан\-да--Ки\-рил\-ло\-ва $\GK(A)$. Для алгебры
$l$-порожд\"{е}нных общих матриц порядка $n$ данная размерность, равна
$(l-1)n^2+1$ (см. \cite{Procesi, Bel04}). 

В то же время Амицур и Левицкий в 1950 году доказали следующий факт:

\begin{theorem}[Амицур--Левицкий, \cite{AL50}] \index{Тождество!стандартное}\index{Теорема!Амицура--Левицкого}
Пусть $S_n$ --- группа перестановок. Определим {\em стандартное тождество} $S_n(x_1,\dots, x_{n}) = 0$ следующим образом: $S_n(x_1,\dots, x_{n}) := \sum\limits_{\sigma\in S_n}(-1)^\sigma  x_{\sigma(1)} x_{\sigma(2)} x_{\sigma(n)}$.
Тогда для любого натурального числа $d$ и любого коммутативного кольца $\FF$ в матричной алгебре $\Mat_d(\FF)$ выполняется стандартное тождество $S_{2d}=0$.
\end{theorem}

Минимальная степень тождества алгебры
$l$-порожд\"{е}нных общих матриц порядка $n$ равна $2n$ согласно теореме Ами\-цу\-ра-ыЛе\-виц\-ко\-го. Имеет место следующее:

\begin{proposition}
Высота $l$-порожд\"{е}нной $\PI$-алгебры степени $n$, а также
множества  не $n$-раз\-би\-ва\-ем\-ых слов над $l$-бук\-вен\-ным
алфавитом, не менее, чем $(l-1)n^2/4+1$.
\end{proposition}

Другие примеры использования теоремы Ами\-цу\-ра--Ле\-виц\-ко\-го можно найти в работе \cite{FI13}

Нижние оценки на индекс нильпотентности были установлены
Е.~Н.~Кузьминым в работе \cite{Kuz75}.  Е.~Н.~Кузьмин привел
пример $2$-порожд\"{е}нной алгебры с тождеством $x^n=0$, индекс
нильпотентности которой строго больше $(n^2+n-2)/2$. Вопрос нахождения нижних оценок рассматривается в главе \ref{ch:pu_number} (см. также \cite{Kh11}).

В то же время для случая нулевой характеристики и счетного числа
образующих Ю.~П.~Размыслов (см. например, \cite{Razmyslov3})
получил верхнюю оценку на индекс нильпотентности, равную $n^2$.

\newpage

\chapter{Оценки индекса нильпотентности конечно порожд\" {е}нных алгебр с ниль-тождеством}\label{ch:nil}
\section{Оценки на появление степеней подслов}
\subsection{План доказательства субэкспоненциальности индекса нильпотентности}

В леммах \ref{Lm0.1}, \ref{c:lem1.2} и \ref{c:lem1.3} описываются достаточные условия для присутствия периода длины $d$ в не $n$-разбиваемом слове $W$. В лемме \ref{c:lem1.4} связываются понятия $n$-разбиваемости слова $W$ и множества его хвостов. После этого определ\"{е}нным образом выбирается подмножество множества хвостов слова $W$, для которого можно применить теорему Дилуорса. Затем мы раскрашиваем хвосты и их первые буквы в соответствии с принадлежностью к цепям, полученным при применении теоремы Дилуорса.

Необходимо изучить, в какой позиции начинают отличаться соседние хвосты в каждой цепи. Вызывает интерес, с какой $``$частотой$"$ эта позиция попадает в $p$-хвост для некоторого $p\leqslant n$. Потом мы несколько обобщаем наши рассуждения, деля хвосты на сегменты по несколько букв, а затем рассматривая, в какой сегмент попадает позиция, в которой начинают отличаться друг от друга соседние хвосты в цепи.
В лемме \ref{c:lem2} связываются рассматриваемые $``$частоты$"$ для $p$-хвостов и $kp$-хвостов для $k = 3$.

В завершение доказательства строится иерархическая структура на основе применения леммы \ref{c:lem2}, т.~е. рассматриваем сначала сегменты $n$-хвостов, потом подсегменты этих сегментов и т.~д. Далее рассматривается наибольшее возможное количество хвостов из подмножества, для которого была применена теорема Дилуорса, после чего оценивается сверху общее количество хвостов, а значит, и букв слова $W$.

\subsection{Свойства периодичности и $n$-разбиваемости}\label{s:nil:basic_properties}

\smallskip


Пусть $\{a_1, a_2,\ldots ,a_l\}$ --- алфавит, над которым проводится построение слов. Порядок $a_1\prec a_2\prec\dots\prec a_l$ индуцирует
лексикографический порядок на словах над заданным алфавитом. Для удобства введ\"{е}м следующие определения:

\begin{definition}
а) Если в слове $v$ содержится подслово вида $u^t,$ то будем
говорить, что в слове $v$ содержится период длины $t.$

б) Если слово $u$ является началом слова $v$, то такие слова \index{Слова!несравнимые}
называют {\em несравнимыми} и обозначают $u\approx v$.

в) Слово $v$ --- {\em хвост} слова $u$, если найдется слово $w$ \index{Слова!хвост}
такое, что $u=wv$.

г) Слово $v$ --- {\em $k$-хвост} слова $u$, если $v$ состоит из \index{Слова!$k$-хвост}
$k$ первых букв некоторого хвоста $u$. Если в хвосте $u$ меньше $k$ букв, то считаем $v=u$.

г $'$) {\em $k$-начало} то же самое, что и $k$-хвост. \index{Слова!$k$-начало}

д) Пусть слово $u$ {\em левее} слова $v$, если начало слова $u$ левее начала слова $v$.
\end{definition}

\begin{notation}
а) Для вещественного числа $x$ положим $\ulcorner x\urcorner := -[-x].$

б) Обозначим как $|u|$ длину слова $u$.
\end{notation}

Для доказательства потребуются следующие достаточные условия наличия периода:

\begin{lemma}       \label{Lm0.1}
В слове $W$ длины $x$ либо первые $[x/d]$ хвостов попарно
сравнимы, либо в слове $W$ найдется период длины $d$.
\end{lemma}

\begin{proof}  Пусть в слове $W$ не нашлось слова вида $u^{d}$. Рассмотрим первые
$[x/{d}]$ хвостов. Предположим, что среди них нашлись 2
несравнимых хвоста $v_1$ и $v_2$. Пусть $v_1=u\cdot v_2$. Тогда
$v_2=u\cdot v_3$ для некоторого $v_3$. Тогда $v_1=u^2\cdot v_3$.
Применяя такие рассуждения, получим, что $v_1=u^{d}\cdot
v_{{d}+1}$, так как $|u|<x/ {d}$, $|v_2|\geqslant ({d}-1)x/ {d}$.
Противоречие.\end{proof}

\begin{lemma}     \label{c:lem1.2}
Если в слове $V$ длины $k\cdot t$ не больше $k$ различных подслов
длины $k$, то $V$ включает в себя период длины $t$.
\end{lemma}

\begin{proof} Докажем лемму индукцией по $k$. База при $k = 1$ очевидна. Если
находится не больше, чем $(k - 1)$ различных подслов длины $(k -
1)$, то применяем индукционное предположение. Если существуют $k$
различных подслов длины $(k - 1)$, то каждое подслово длины $k$
однозначно определяется своими первыми $(k - 1)$ буквами. Значит,
$V=v^t$, где $v$ --- $k$-хвост $V$.\end{proof}

\begin{definition} \index{Слово!$n$-разбиваемое!в обычном смысле}
а) Слово $W$ --- \textit{$n$-разбиваемо в обычном смысле}, если
найдутся $u_1, u_2,\ldots,u_n$ такие, что $W=v\cdot u_1\cdots
u_n$, при этом $u_1\succ\ldots \succ u_n$. Слова $u_1,
u_2,\dots, u_n$ назов\"{е}м {\em $n$-разбиением} слова $W$.

б) В текущем доказательстве слово $W$ будем называть \index{Слово!$n$-разбиваемое!в хвостовом смысле}
\textit{$n$-разбиваемым в хвостовом смысле}, если найдутся хвосты
$u_1,\ldots,u_n$ такие, что $u_1\succ u_2\succ \ldots \succ u_n$ и
для любого $i=1, 2,\ldots, n - 1$ начало $u_i$ слева от начала
$u_{i+1}$. Хвосты $u_1,
u_2,\dots, u_n$ назов\"{е}м {\em $n$-разбиением в хвостовом смысле} слова $W$. В части \ref{ch:nil}, если не оговорено противное, то под
\textit{$n$-разбиваемыми} словами мы подразумеваем $n$-разбиваемые \index{$n$-разбиение в хвостовом смысле}
в хвостовом смысле.

в) Слово $W$ --- \textit{$(n,d)$-сократимое}, если оно либо \index{Слово!$(n,d)$-сократимое}
$n$-разбиваемо в обычном смысле, либо находится слово вида $u^{d}\subseteq W$.
\end{definition}

Теперь опишем достаточное условие $(n,d)$-сократимости и его связь с $n$-разбиваемостью.

\begin{lemma}          \label{c:lem1.3}
Если в слове $W$ найдутся $n$ одинаковых непересекающихся подслов
$u$ длины $n\cdot{d}$, то $W$ --- $(n,d)$-сократимое.
\end{lemma}

\begin{proof} Предположим противное. Рассмотрим хвосты
$u_1,u_2,\ldots,u_n$ слова $u$, которые начинаются с каждой из его
первых $n$ букв. Перенумеруем хвосты так, чтобы выполнялись
неравенства: $u_1~\succ\ldots\succ~u_n.$ Из леммы \ref{Lm0.1} они
несравнимые. Рассмотрим подслово $u_1,$ лежащее в самом левом экземпляре слова $u,$ подслово $u_2$ --- во втором слева, $\dots$, $u_n$ --- в $n$-ом слева. Получили
$n$-разбиение слова $W$. Противоречие.\end{proof}

\begin{proposition}\label{pr:0}
Если для некоторых слов $u, v, w$ верно соотношение $|u|\leqslant |v| < |w|$ и, кроме того, $u\approx w, v\approx w$, то слова $u$ и $v$ несравнимы.
\end{proposition}

\begin{remark}
Если для некоторого действительного числа $a$ мы говорим про $a$-разбиваемость, то имеется ввиду $[a]$-разбиваемость.
\end{remark}

\begin{lemma}      \label{c:lem1.4}
Если слово $W$ является $[{3\over 2}(n + 1)d(\log_3{(nd)} + 2)]$-разбиваемым, то оно --
$(n,d)$-сократимое.
\end{lemma}

\begin{notation}
$p_{n,d}:=[{3\over 2}(n + 1)d(\log_3{(nd)} + 2)].$
\end{notation}

\begin{proof}
От противного. Пусть слово $W$ является $p_{n,d}$-разбиваемым, хвосты $u_1\prec u_2\prec\dots\prec u_{p_{n,d}}$ образуют $p_{n,d}-$разбиение, но слово $W$ --- не $(n,d)$-сократимое.

\begin{notation}
Пусть для $1\leqslant i < p_{n,d}$ слово $v_i$ --- подслово слова $W$, которое начинается первой буквой хвоста $u_i$ и заканчивается буквой, стоящей на одну позицию левее первой буквы хвоста $u_{i + 1}$. Также считаем, что $v_{p_{n,d}} = u_{p_{n,d}}$.
\end{notation}

 Предположим, что для любого числа $1\leqslant i\leqslant [{3\over 2}(n - 1)d(\log_3{(nd)} + 2)]$ найдутся числа $0\leqslant j_i\leqslant k_i < m_i\leqslant q_i < [{3}d(\log_3{(nd)} + 2)]$ такие, что $\prod\limits_{s = j_i+i}^{k_i+i} v_s\prec \prod\limits_{s = m_i+i}^{q_i+i} v_s$. Тогда рассмотрим последовательность чисел $i_s = {2s\over n}p_{n,d}+1$, где $0\leqslant s\leqslant [{n-1\over 2}]$. Тогда последовательность слов
$\prod\limits_{s = j_{i_0}+i_0}^{k_{i_0}+i_0} v_s\prec\prod\limits_{s = m_{i_0}+i_0}^{q_{i_0}+i_0} v_s\prec \prod\limits_{s = j_{i_1}+i_1}^{k_{i_1}+i_1} v_s\prec\prod\limits_{s = m_{i_1}+i_1}^{q_{i_1}+i_1} v_s\prec \prod\limits_{s = j_{i_2}+i_2}^{k_{i_2}+i_2} v_s\prec\dots\prec
\prod\limits_{s = j_{i_{[{n-1\over 2}]}}+i_{[{n-1\over 2}]}}^{k_{i_{[{n-1\over 2}]}}+i_{[{n-1\over 2}]}} v_s\prec\prod\limits_{s = m_{i_{[{n-1\over 2}]}}+i_{[{n-1\over 2}]}}^{q_{i_{[{n-1\over 2}]}}+i_{[{n-1\over 2}]}} v_s$ образует $2{[{n+1\over 2}]}-$разбиение в обычном смысле слова $W$.

Значит, слово $W$ --- $(n,d)$-сократимо. Противоречие.

Следовательно, найд\"{е}тся такое число $1\leqslant i\leqslant [{3\over 2}(n - 1)d(\log_3{(nd)} + 2)]$, что для любых $0\leqslant j\leqslant k< m\leqslant q < [{3}d(\log_3{(nd)} + 2)]$ имеем $\prod\limits_{s = j+i}^{k+i} v_s\approx\prod\limits_{s = m+i}^{q+i} v_s$.

Без ограничения общности $i = 1$. Для некоторого натурального $t$ рассмотрим некоторую последовательность натуральных чисел $\{k_i\}_{i=1}^{t}$ такую, что $k_1=3$ и $\prod\limits_{i=2}^{t}k_i>nd$. В силу леммы \ref{c:lem1.3} $\inf\limits_{0<j\leqslant k_1 d}|v_j|\leqslant nd$. Пусть $\inf\limits_{0<j\leqslant k_1 d}|v_j|$ достигается на $v_{j_1}$, где $0<j_1\leqslant k_1 d$ (если таких минимумов несколько, то бер\"{е}м самый правый из них).
\begin{itemize}
    \item Если $j_1\leqslant d$, то $d|v_{j_1}|-$начала хвостов $u_{j_1}$ и $u_{j_1 + 1}$ не пересекаются и несравнимы со словом $\prod\limits_{j=2d+1}^{3d}v_j$ и меньше его по длине. Следовательно, по предложению \ref{pr:0} $d|v_j|-$начала хвостов $u_{j_1}$  и $u_{j_1+1}$ несравнимы, а, значит, $v_{j_1}^d$ --- подслово слова $W$.
    \item Если $d<j_1\leqslant 2d,$ то $d|v_j|-$начала хвостов $u_{j_1}$  и $u_{j_1+1}$ не пересекаются со словом $v_{3d+1}$ и несравнимы со словом $\prod\limits_{j=1}^{d}v_j$. Кроме того, эти $d|v_j|-$начала меньше его по длине. Значит, по предположению \ref{pr:0}, $v_{j_1}^d$ --- подслово слова $W$.
    \item Если $2d<j_1\leqslant 3d$ и $d|v_j|-$начало хвоста $u_{j_1+1}$ не пересекается с $v_{[{3}d(\log_3{(nd)} + 2)]}$, то, аналогично предыдущему случаю, $v_{j_1}^d$ --- подслово слова $W$. Пусть $d|v_j|-$начало хвоста $u_{j_1+1}$ пересекается с $v_{[{3}d(\log_3{(nd)} + 2)]}$. Тогда для некоторого натурального числа $t$, которое будет выбрано позднее, рассмотрим последовательность натуральных чисел $\{k_j\}_{j=2}^t$, для которой $\prod\limits_{j=2}^t\geqslant nd$. Выберем $t$ так, что $\sum\limits_{j=2}^{t}\leqslant (\log_3{(nd)} + 1)$. Можно показать, что такое $t$ всегда существует. Заметим, что слово $\prod\limits_{j=k_1 d + 1}^{(k_1 + k_2)d}v_j$ содержится в $d|v_{j_1}|-$начале хвоста $u_{j_1} + 1$.
        Рассмотрим $v_{j_2}$ --- слово наименьшей длины среди слов $v_j$ при $k_1 d < j\leqslant (k_1 + k_2)d$. Если таких слов несколько, то в качестве наименьшего возьм\"{е}м самое правое из них. Тогда $|v_{j_2}|\leqslant [{|v_{j_1}|\over k_2}]$. Теперь рассмотрим $d|v_{j_2}|-$начало хвоста $u_{j_2 + 1}$. Если оно не пересекается со словом $v_{[{3}d(\log_3{(nd)} + 2)]}$, то $v_{j_2}^d$ --- подслово слова $W$. В противном случае слово $\prod\limits_{j=(k_1 + k_2) d + 1}^{(k_1 + k_2 + k_3)d}v_j$ является подсловом $d|v_{j_2}|-$начала хвоста $u_{j_2 + 1}$.
        
        Пусть для некоторого числа $i$ такого, что $2\leqslant i < t$ среди чисел $j$ таких, что $d\sum\limits_{s=1}^{i-1}k_s<j\leqslant d\sum\limits_{s=1}^{i}k_s$ найд\"{е}тся число $j_i$ такое, что $|v_{j_i}|\leqslant [{|v_{j_1}|\over\prod\limits_{s=2}^{i}k_s}]$. Рассмотрим $d|v_{j_i}|-$начало хвоста $u_{j_i + 1}$. Если оно не пересекается со словом $v_{[{3}d(\log_3{(nd)} + 2)]}$, то $v_{j_i}^d$ --- подслово слова $W$. В противном случае слово $\prod\limits_{j=d\sum\limits_{s=1}^{i}k_s + 1}^{d\sum\limits_{s=1}^{i+1}k_s}v_j$ является подсловом $d|v_{j_i}|-$начала хвоста $u_{j_i + 1}$.
         Рассмотрим $v_{j_{i+1}}$ --- слово наименьшей длины среди слов $v_j$ при $d\sum\limits_{s=1}^{i}k_s < j\leqslant d\sum\limits_{s=1}^{i+1}k_s$. Тогда $|v_{j_{i+1}}|\leqslant [{|v_{j_i}|\over k_{i+1}}]\leqslant [{|v_{j_1}|\over\prod\limits_{s=2}^{i+1}k_s}]$.
         Таким образом, $|v_t|\leqslant [{|v_{j_1}|\over\prod\limits_{s=2}^{t}k_s}] < 1.$ Получено противоречие, из которого и вытекает утверждение леммы.

\end{itemize}
\end{proof}





Пусть $W$ --- не $(n,d)$-сократимое слово. Рассмотрим $U$ --- $[\left|
W\right|/d]$-хвост слова $W$. Тогда $W$ --- не $(p_{n, d}+1)$-разбиваемое.
Пусть $\Omega$ --- множество хвостов слова $W$, которые начинаются \index{Слова!хвост}
в $U$. Тогда по лемме \ref{Lm0.1} любые два элемента из $\Omega$
сравнимы. Естественным образом строится биекция между $\Omega$,
буквами $U$ и натуральными числами от $1$ до
$\left|\Omega\right|=\left|U\right|$.

Введем слово $\theta$ такое, что $\theta$ лексикографически меньше
любого слова.

\begin{remark}
В текущем доказательстве теоремы \ref{c:main1} все хвосты мы
предполагаем лежащими в $\Omega$.
\end{remark}

\section{Оценки на появление периодических фрагментов}

\subsection{Применение теоремы Дилуорса.} Для хвостов $u$ и $v$ положим \index{Теорема!Дилуорса}
$u<v$, если $u \prec v$ и, кроме того, $u$ левее $v$. Тогда по
теореме Дилуорса $\Omega$ можно разбить на $p_{n, d}$ цепей, где в
каждой цепи $u \prec v$, если $u$ левее $v$. Покрасим начальные
позиции хвостов в $p_{n, d}$ цветов в соответствии с принадлежностью к
цепям. Фиксируем натуральное число $p$. Каждому
натуральному числу $i$ от 1 до $\left|\Omega\right|$ сопоставим
$B^p(i)$ --- упорядоченный набор из $p_{n, d}$ слов $\{f(i,j)\},$
построенных по следующему правилу:

{\it Для каждого $j = 1, 2,\ldots, p_{n, d}$\ положим

$$f(i,j)=\left\{\max \
f\leqslant i: f\ \mbox{раскрашено\ в\ цвет}~j\right\}.$$

Если такого $f$ не найдется, то слово из $B^p(i)$ на позиции
$j$ считаем равным $\theta$, в противном случае это слово считаем
равным $p$-хвос\-ту, который начинается с $f(i,j)$-ой
буквы. }

Неформально говоря, мы наблюдаем, с какой скоростью хвосты ``эволюционируют'' в своих цепях, если рассматривать последовательность позиций слова $W$ как ось времени.

\subsection{Наборы $B^p(i)$, процесс на позициях} \label{ss:set_bpi}

\begin{lemma}[О процессе]   \label{c:lem} \index{Лемма!о процессе}
Дана последовательность $S$ длины $|S|$, составленная из слов
длины $(k-1)$. Каждое из них состоит из $(k-2)$ символа $``0"$ и одной
$``1"$. Пусть $S$ удовлетворяет следующему условию:
{\em если для некоторого $0 < s \leqslant k-1$ найдутся $p_{n, d}$
слов, в которых $``1"$ стоит на $s$-ом месте, то между первым и
$p_{n, d}$-ым из этих слов найдется слово, в котором $``1"$ стоит
строго меньше, чем на $s$-ом месте}; $L(k-1)=\sup\limits_S |S|$.

Тогда $L(k-1)\leqslant p_{n, d}^{k-1}-1$.
\end{lemma}

\begin{proof} $L(1) \leqslant p_{n, d}-1$. Пусть $L(k-1) \leqslant {p_{n, d}^{k-1}} -
1$. Покажем, что $L(k) \leqslant {p_{n, d}^{k}} - 1$.
Рассмотрим слова, у которых символ $``1"$ стоит на первом месте.
Их не больше $p_{n, d}-1$. Между любыми двумя из них, а также перед
первым и после последнего, количество слов не больше $L(k-1)
\leqslant {p_{n, d}^{k-1}} - 1$. Получаем, что
$$
L(k) \leqslant p_{n, d} - 1 +
\left(p_{n, d}\right)\left({\left(p_{n, d}\right)^{k-1}}-1\right) =
{\left(p_{n, d}\right)^{k}}-1$$
\end{proof}

Нам требуется ввести некоторую величину, которая бы численно оценивала скорость ``эволюции'' наборов $B^p(i)$:

\begin{definition}
Положим $$\psi(p):= \left\{\max \ k:
B^p(i)=B^p(i+k-1)\right\}.$$
В частности, по лемме \ref{c:lem1.2},
$\psi(p_{n,d})\leqslant p_{n, d} d$.

Для заданного $\alpha$ определим разбиение последовательности первых $\left|\Omega\right|$ позиций
${i}$ слова $W$ на классы эквивалентности $\AC_\alpha$ следующим образом: $i
\AC_\alpha j$, если $B^\alpha(i)=B^\alpha(j).$
\end{definition}

\begin{proposition}
Для любых натуральных $a<b$ имеем $\psi(a) \leqslant \psi(b).$
\end{proposition}

\begin{lemma}[Основная] \label{c:lem2} \index{Лемма!основная!для ниль-случая}
Для любых натуральных чисел $a, k$ верно неравенство
$$\psi(a)\leqslant p_{n,d}^k\psi(k\cdot a)+k\cdot a$$
\end{lemma}

\begin{proof}
Рассмотрим по наименьшему представителю из каждого класса $\AC_{k\cdot a}$. Получена последовательность позиций $\{i_j\}.$ Теперь рассмотрим все $i_j$ и $B^{k\cdot a}(i_j)$ из одного класса эквивалентности по $\AC_a.$ Пусть он состоит из $B^{k\cdot a}(i_j)$ при $i_j\in[b, c).$ Обозначим за $\{i_j\}'$ отрезок последовательности $\{i_j\},$ для которого $i_j\in[b, c-k\cdot a).$

Фиксируем некоторое натуральное число $r, 1\leqslant r\leqslant p_{n,d}.$ Назов\"{е}м все $k\cdot a$-начала цвета $r$, начинающиеся с позиций слова $W$ из $\{i_j\}'$, представителями типа $r$. Все представители типа $r$ будут попарно различны, так как они начинаются с наименьших позиций в классах эквивалентности по $\AC_{k\cdot a}.$ Разобь\"{е}м каждый представитель типа $r$ на $k$ сегментов длины $a$. Пронумеруем сегменты внутри каждого представителя типа $r$ слева направо числами от нуля до $(k-1)$. Если найдутся $(p_{n,d}+1)$ представителей типа $r$, у которых совпадают первые $(t-1)$ сегментов, но которые попарно различны в $t$-ом, где $t$ --- натуральное число, $1\leqslant t\leqslant k-1$, то найдутся две первых буквы $t$-го сегмента одного цвета. Тогда позиции, с которых начинаются эти сегменты, входят в разные классы эквивалентности по $\AC_a.$

Применим лемму \ref{c:lem} следующим образом: во всех представителях типа $r$, кроме самого правого, будем считать сегменты {\it единичными}, если именно в них находится наименьшая позиция, в которой текущий представитель типа $r$ отличается от предыдущего. Остальные сегменты считаем {\it нулевыми.}

Теперь можно применить лемму о процессе с параметрами, совпадающими с заданными в условии леммы. Получаем, что в последовательности $\{i_j\}'$ будет не более $p_{n,d}^{k-1}$ представителей типа $r$. Тогда в последовательности $\{i_j\}'$  будет не более $p_{n,d}^{k}$ членов. Таким образом, $c-b\leqslant p_{n,d}^k\psi(k\cdot a)+k\cdot a.$ \end{proof}

\subsection{Завершение доказательства субэкспоненциальности индекса нильпотентности}

Пусть
$$a_0 = 3^{\ulcorner \log_3 p_{n,d}\urcorner}, a_1 = 3^{\ulcorner \log_3 p_{n,d}\urcorner-1},\ldots,a_{{\ulcorner \log_3 p_{n,d}\urcorner}} =1.$$
При этом $\left|W\right|\leqslant d\left|\Omega\right| + d$ в силу леммы \ref{Lm0.1}.

Так как набор $B^1(i)$ принимает не более $(1+p_{n,d}l)$ различных значений, то $\left|W\right|\leqslant d(1+p_{n,d}l)\psi(1) + d.$

По лемме \ref{c:lem2}
$$\psi(1)< (p_{n,d}^3 + p_{n,d})\psi(3)<(p_{n,d}^3 + p_{n,d})^2\psi(9)<\cdots <(p_{n,d}^3 + p_{n,d})^{\ulcorner \log_3 p_{n,d}\urcorner}\psi(p_{n,d})\leqslant
$$
$$\leqslant (p_{n,d}^3 + p_{n,d})^{\ulcorner \log_3 p_{n,d}\urcorner}p_{n,d}d
$$

Подставляя $p_{n,d} = 4nd-1,$ получаем
$$\left|W\right|< 4^{5+3\log_3 4} l (nd)^{3 \log_3 (nd)+(5+6\log_3 4)}d^2.
$$

Отсюда имеем {\bf утверждение теоремы \ref{c:main1}.}

Доказательство теоремы \ref{t2:log_2} завершается так же, только вместо последовательности
$$a_0 = 3^{\ulcorner \log_3 p_{n,d}\urcorner}, a_1 = 3^{\ulcorner \log_3 p_{n,d}\urcorner-1},\ldots,a_{{\ulcorner \log_3 p_{n,d}\urcorner}} =1$$
рассматривается последовательность
$$a_0 = 2^{\ulcorner \log_2 p_{n,d}\urcorner}, a_1 = 2^{\ulcorner \log_2 p_{n,d}\urcorner-1},\ldots,a_{{\ulcorner \log_2 p_{n,d}\urcorner}} =1.$$

\newpage

\chapter{Оценки высоты и существенной высоты конечно порожд\"{е}нной $\PI$-алгебры.}\label{ch:height}

\section{Оценка существенной высоты.}\label{ch:height:s:ess}
В данном разделе мы продолжаем доказывать основную теорему \ref{c:main2}.
Попутно доказывается теорема \ref{ThThick}. Будем смотреть на позиции букв слова $W$ как на ось времени, то есть подслово $u$ встретилось раньше подслова $v$, если $u$ целиком лежит левее $v$ внутри слова $W$.



\subsection{Нахождение различных периодических фрагментов в слове} \label{c:sub1}

Обозначим за $s$ количество подслов слова $W$ с периодом длины меньше $n$, в которых период повторяется больше $2n$ раз и которые попарно разделены сравнимыми с предыдущим периодом подсловами длины больше $n$. Пронумеруем их от начала к концу слова: $x_1^{2n}, x_2^{2n},\ldots,x_s^{2n}$. Таким образом $W=y_0x_1^{2n}y_1x_2^{2n}\cdots x_s^{2n}y_s.$



Если найд\"{е}тся $i$ такое, что слово $x_i$ длины не меньше $n$, то в слове $x_i^2$ найдутся $n$ попарно сравнимых хвостов, а значит, слово $x_i^{2n}$-- $n$-разбиваемое. Получаем, что число $s$ не меньше, чем существенная высота слова $W$ над множеством слов длины меньше $n.$

\begin{definition} \index{Слово!нециклическое}
Слово $u$ назовем {\em нециклическим}, если $u$ нельзя представить
в виде $v^k$, где $k>1$.
\end{definition}

\begin{definition} \index{Слово!-цикл}
{\em Слово-цикл $u$} --- слово $u$ со всеми его сдвигами по циклу.
\end{definition}
\begin{definition} \index{Слово!циклическое}
{\em Циклическое слово $u$} --- цикл из букв слова $u$, где после его последней буквы ид\"{е}т первая.
\end{definition}
\begin{definition} \index{Слова!сравнимые!сильно}
Если любые два циклических сдвига слов $u$ и $v$ сравнимы, то назов\"{е}м слова $u$ и $v$ {\em сильно сравнимыми.}
Аналогично определяется сильная сравнимость слово-циклов и циклических слов.
\end{definition}

Далее мы будем пользоваться естественной биекцией между слово-циклами и циклическими словами.

\begin{definition} \index{Слово!$n$-разбиваемое!сильно}
Слово $W$ называется {\em сильно $n$-разбиваемым}, если его 
можно представить в виде $W=W_0W_1\cdots W_n$, где подслова
$W_1,\dots,W_n$ идут в порядке лексикографического убывания, и
каждое из слов $W_i, i=1, 2,\ldots, n$ начинается с некоторого
слова $z_i^k\in Z$, все $z_i$ различны.
\end{definition}

\begin{lemma}\label{lem4.10}
Если найд\"{е}тся число $m, 1\leqslant m<n,$ такое, что существуют $(2n-1)$ попарно несравнимых слов длины $m: x_{i_1},\ldots ,x_{i_{2n-1}},$ то $W$ --- $n$-разбиваемое.
\end{lemma}
 \begin{proof}Положим $x:=x_{i_1}.$ Тогда в слове $W$ найдутся непересекающиеся подслова\linebreak[1] $x^{p_1}v'_1,\ldots ,x^{p_{2n-1}}v'_{2n-1},$ где $p_1,\ldots ,p_{2n-1}$ --- некоторые натуральные числа, большие $n,$ а $v'_1,\ldots ,v'_{2n-1}$ --- некоторые слова длины $m,$ сравнимые с $x, v'_1=v_{i_1}.$ Тогда среди слов $v'_1,\ldots ,v'_{2n-1}$ найдутся либо $n$ лексикографически больших $x$, либо $n$ лексикографически меньших $x$. Можно считать, что $v'_1,\ldots ,v'_n$ --- лексикографически больше $x$. Тогда в слове $W$ найдутся подслова $v'_1, xv'_2,\ldots ,x^{n-1}v'_n,$ идущие слева направо в порядке лексикографического убывания.\end{proof}

Рассмотрим некоторое число $m, 1\leqslant m<n.$ Разобь\"{е}м все $x_i$ длины $m$ на эквивалентности по сильной несравнимости и выберем по одному представителю из каждого класса эквивалентности. Пусть это слова $x_{i_1},\ldots ,x_{i_s'},$
где $s'$ --- некоторое натуральное число. Так как подслова $x_i$ являются периодами, будем рассматривать их как слово-циклы.

\begin{notation}

$v_k := x_{i_k}$

Пусть $v(k, i)$, где $i$ --- натуральное число от 1 до $m$, --- циклический сдвиг слова $v_k$ на $(k - 1)$ позиций вправо, то есть $v(k, 1) = v_k$, а первая буква слова $v(k, 2)$ является второй буквой слова $v_k$. Таким образом, $\{ v(k, i)\}_{i=1}^m$ --- слово-цикл слова $v_k$. Заметим, что для любых $1\leqslant i_1, i_2\leqslant p, 1\leqslant j_1, j_2\leqslant m$ слово $v(i_1, j_1)$ сильно несравнимо со словом $v(i_2, j_2)$.
\end{notation}


\begin{remark}
Случаи $m = 2, 3, n-1$ более подробно рассмотрены в работах \cite{Kh11, Kh11(2)}.
\end{remark}

\subsection{Применение теоремы Дилуорса}      \label{c:sub2}\index{Теорема!Дилуорса}

Рассмотрим множество $\Omega' = \{ v(i, j)\}$, где $1\leqslant i\leqslant p, 1\leqslant j\leqslant m.$ Введ\"{е}м следующий порядок на словах $v(i, j):$

$v(i_1, j_1)\succ v(i_2, j_2),$ если

1) $v(i_1, j_1)> v(i_2, j_2)$

2) $i_1 > i_2$

\begin{lemma}\label{c:lem4}
Если во множестве $\Omega'$ для порядка $\succ$ найд\"{е}тся антицепь длины $n$, то слово $W$ будет $n$-разбиваемым.
\end{lemma}
 \begin{proof}Пусть нашлась антицепь длины $n$ из слов $v(i_1, j_1)$, $v(i_2, j_2),\ldots,$ $v(i_n, j_n)$; где $i_1\leqslant i_2\leqslant\cdots\leqslant i_n.$ Если все неравенства между $i_k$ --- строгие, то слово $W$ --- $n$-разбиваемое по определению.

Предположим, что для некоторого числа $r$ нашлись $i_{r+1} = \cdots = i_{r+k}$, где либо $r = 0$, либо $i_r < i_{r+1}$. Кроме того, $k$ --- такое натуральное число, что
либо $k = n - r$, либо $i_{r+k} < i_{r+k+1}$.

Слово $s_{i_{r+1}}$ --- периодическое, следовательно, оно представляется в виде произведения $n$ экземпляров слова $v^2_{i_{r+1}}$. Слово $v^2_{i_{r+1}}$ содержит слово-цикл $v_{i_{r+1}}$. Значит, в слове $s_{i_{r+1}}$ можно выбрать непересекающиеся подслова, идущие в порядке лексикографического убывания, равные $v(i_{r+1}, j_{r+1}),\ldots,v(i_{r+k}, j_{r+k})$ соответственно. Таким же образом поступаем со всеми множествами равных индексов  в последовательности $\{i_r\}_{r=1}^n$. Получаем $n$-разбиваемость слова $W$. Противоречие.\end{proof}

Значит, множество $\Omega'$ можно разбить на $(n-1)$ цепь.

\begin{notation}
Положим $q_n = (n-1)$.
\end{notation}

\subsection{Наборы $C^\alpha(i)$, процесс на позициях}

Покрасим первые буквы слов из $\Omega'$ в $q_n$ цветов в соответствии с
принадлежностью к цепям. Покрасим также числа от $1$ до $\left|\Omega '
\right|$ в соответствующие цвета. Фиксируем натуральное число
$\alpha\leqslant m$. Каждому числу $i$ от 1 до $\left|\Omega '
\right|$ сопоставим упорядоченный набор слов $C^\alpha(i)$,
состоящий из $q_n$ слов по следующему правилу:

\medskip
{\it Для каждого $j=1,2,\ldots,q_n$\ положим

$f(i,j)=\{\max \ f\leqslant i:$ существует $k$ такое, что $v(f,k)$ раскрашено в цвет $j$ и
$\alpha$-хвост, который начинается с $f$, состоит только из
букв, являющихся первыми буквами хвостов из $\Omega '\}$.

Если такого $f$ не найд\"{е}тся, то слово из $C^\alpha(i)$ считаем
равным $\theta$, в противном случае это слово считаем равным
$\alpha$-хвосту слова $v(f,k)$. }

\begin{notation}
Положим $\phi(a)$ равным наибольшему $k$ такому, что найд\"{е}тся число $i$, для которого верно равенство $C^a(i)=C^a(i+k-1)$.

Для заданного $a\leqslant m$ определим разбиение последовательности
слово-циклов $\{i\}$ слова $W$ на классы эквивалентности следующим
образом: $i \AC_a j,$ если $C^a(i) = C^a(j)$.
\end{notation}

Заметим, что построенная конструкция во многом аналогична построенной в доказательстве теоремы \ref{c:main1}. Можно обратить внимание на схожесть $B^a(i)$ и $C^a(i)$, а также $\psi(a)$ и $\phi(a)$.

\begin{lemma}\label{lem:m}
$\phi(m) \leqslant q_n/m$.
\end{lemma}

 \begin{proof}Напомним, что слово-циклы были пронумерованы. Рассмотрим слово-циклы с номерами $i, i + 1,\ldots ,i + [q_n/m].$ Ранее было показано, что каждый слово-цикл состоит из $m$ различных слов. Рассмотрим теперь слова в слово-циклах $i, i + 1,\ldots ,i + [q_n/m]$ как элементы множества $\Omega'.$ При таком рассмотрении у первых букв из слово-циклов появляются свои позиции. Всего рассматриваемых позиций не меньше $n.$ Следовательно, среди них найдутся две позиции одного цвета. Тогда в силу сильной несравнимости слово-циклов имеем утверждение леммы.\end{proof}

\begin{proposition}
Для любых натуральных $a<b$ имеем $\phi(a) \leqslant \phi(b).$
\end{proposition}

\begin{lemma}[Основная] \label{lem:thick} \index{Лемма!основная!для общего случая}
Для натуральных чисел $a, k$ таких, что $ak\leqslant m$, верно неравенство
$$\phi(a)\leqslant p_{n,d}^k\phi(k\cdot a).$$
\end{lemma}

\begin{proof}
Рассмотрим по наименьшему представителю из каждого класса $\AC_{k\cdot a}$. Получена последовательность позиций $\{i_j\}.$ Теперь рассмотрим все $i_j$ и $C^{k\cdot a}(i_j)$ из одного класса эквивалентности по $\AC_a.$ Пусть он состоит из $C^{k\cdot a}(i_j)$ при $i_j\in[b, c).$ Обозначим за $\{i_j\}'$ отрезок последовательности $\{i_j\},$ для которого $i_j\in[b, c).$

Фиксируем некоторое натуральное число $r, 1\leqslant r\leqslant q_n.$ Назов\"{е}м все $k\cdot a$-начала цвета $r$, начинающиеся с позиций слова $W$ из $\{i_j\}'$, представителями типа $r$. Все представители типа $r$ будут попарно различны, так как они начинаются с наименьших позиций в классах эквивалентности по $\AC_{k\cdot a}.$ Разобь\"{е}м каждый представитель типа $r$ на $k$ сегментов длины $a$. Пронумеруем сегменты внутри каждого представителя типа $r$ слева направо числами от нуля до $(k-1)$. Если найдутся $(q_n+1)$ представителей типа $r$, у которых совпадают первые $(t-1)$ сегментов, но которые попарно различны в $t$-ом, где $t$ --- натуральное число, $1\leqslant t\leqslant k-1$, то найдутся две первых буквы $t$-го сегмента одного цвета. Тогда позиции, с которых начинаются эти сегменты, входят в разные классы эквивалентности по $\AC_a.$

Применим лемму \ref{c:lem} следующим образом: во всех представителях типа $r$, кроме самого правого, будем считать сегменты {\it единичными}, если именно в них находится наименьшая позиция, в которой текущий представитель типа $r$ отличается от предыдущего. Остальные сегменты будем считать {\it нулевыми.}

Теперь мы можем применить лемму о процессе с параметрами, совпадающими с заданными в условии леммы. Получаем, что в последовательности $\{i_j\}'$ будет не более $q_n^{k-1}$ представителей типа $r$. Тогда в последовательности $\{i_j\}'$  будет не более $q_n^{k}$ членов. Таким образом, $c-b\leqslant q_n^k\phi(k\cdot a).$ \end{proof}

\subsection{Завершение доказательства субэкспоненциальности существенной высоты}

Пусть
$$a_0 = 3^{\ulcorner \log_3 p_{n,d}\urcorner}, a_1 = 3^{\ulcorner \log_3 p_{n,d}\urcorner-1},\ldots,a_{{\ulcorner \log_3 p_{n,d}\urcorner}} =1.$$
Подставляя эти $a_i$ в леммы
\ref{lem:thick} и \ref{lem:m}, получаем, что
$$\phi(1)\leqslant q_n^3\phi(3)\leqslant q_n^9\phi(9)\leqslant\cdots \leqslant q_n^{3\ulcorner \log_3 m\urcorner}\phi(m)\leqslant
$$
$$\leqslant q_n^{3\ulcorner \log_3 m\urcorner+1}. $$

Так как $C_i^{1}$ принимает не более $1+q_n l$ различных значений, то

$$\left|\Omega'\right|< q_n^{3\ulcorner \log_3 m\urcorner+1}(1+q_n l)< n^{3\ulcorner \log_3 n\urcorner+2} l.$$

По лемме \ref{lem4.10} получаем, что количество $x_i$ длины $m$ меньше $2n^{3\ulcorner \log_3 n\urcorner+3} l.$

Имеем, что количество всех $x_i$ меньше $2n^{3\ulcorner \log_3 n\urcorner+4} l.$

То есть $s<2n^{3\ulcorner \log_3 n\urcorner+4} l.$ Таким образом, {\bf теорема \ref{ThThick} доказана.}

\section{Оценка высоты в смысле Ширшова} \label{ch:height:s:sh}

\subsection{План доказательства}
Будем снова под {\em $n$-разбиваемым
словом} подразумевать {\em $n$-раз\-би\-ва\-е\-мое} в обычном
смысле.
Сначала мы находим необходимое количество фрагментов с длиной периода не меньше $2n$ в слове $W$. Это можно сделать, просто разбив слово $W$ на подслова большой длины, к которым применяется теорема \ref{c:main1}. Однако мы можем улучшить оценку, если сначала выделим в слове $W$ периодический фрагмент с длиной периода не менее $4n$, затем рассмотрим $W_1$ --- слово $W$ с ``вырезанным'' периодическим фрагментом $u_1$. У слова $W_1$ выделяем фрагмент с длиной периода не менее $4n$, после чего рассматриваем $W_2$ --- слово $W_1$ с ``вырезанным'' периодическим фрагментом $u_2$. У слова $W_2$ так же вырезаем периодический фрагмент. Далее продолжаем этот процесс, подробнее описанный в алгоритме \ref{c:al}. Затем по вырезанным фрагментам мы восстанавливаем первоначальное слово $W$. После этого показывается, что в слове $W$ подслово $u_i$ чаще всего не является произведением большого количества не склеенных подслов. В лемме \ref{c:lem3} доказывается, что применение алгоритма \ref{c:al} дает необходимое количество подслов слова $W$ с длиной периода не меньше $2n$ среди вырезанных подслов.

\subsection{Суммирование существенной высоты и степени нильпотентности}

До конца главы будем использовать следующее
\begin{notation} \index{Высота!слова}
$\Ht(w)$ --- высота слова $w$ над множеством слов степени не выше $n$.
\end{notation}
Рассмотрим слово $W$ с высотой $\Ht(W)>\Phi(n,l)$. Теперь для него
провед\"{е}м следующий алгоритм:

\begin{algorithm}  \label{c:al}
{\ }

\begin{description}
    \item[Первый шаг.] По теореме \ref{c:main1} в слове $W$ найд\"{е}тся подслово с длиной периода
    $4n$. Пусть $W_0=W=u'_1x_{1'}^{4n}y'_1$, прич\"{е}м слово $x_{1'}$ --- нециклическое.
    Представим $y'_1$ в виде $y'_1=x_{1'}^{r_2}y_1$, где $r_2$ --- максимально возможное
    число. Слово $u'_1$ представим как $u'_1=u_1x_{1'}^{r_1}$, где $r_1$ --- наибольшее возможное. Обозначим за $f_1$ следующее слово:
$$W_0=u_1x_{1'}^{4n+r_1+r_2}y_1=u_1f_1y_1.$$
    Назов\"{е}м позиции, входящие в слово $f_1$, {\em скучными,} последнюю позицию слова $u_1$ --- {\em скучной типа $1$,} вторую с конца позицию $u_1$ --- {\em скучной типа $2$} и так далее, $n$-ую с конца позицию $u_1$ --- {\em скучной типа $n$.} Положим $W_1 = u_1 y_1.$

    \item[$k$-ый шаг.] Рассмотрим слова $u_{k-1},\ y_{k-1},\ W_{k-1}=u_{k-1}y_{k-1}$,
    построенные на предыдущем шаге. Если $|W_{k-1}|\geqslant\Phi(n,l),$ то применим теорему \ref{c:main1} к слову $W$ с тем условием, что процесс в основной лемме \ref{c:lem2} будет вестись только по не скучным позициям и скучным позициям типа больше $ka,$ где $k$ и $a$ --- параметры леммы \ref{c:lem2}.

    Таким образом, в слове $W_{k-1}$ найд\"{е}тся
нециклическое подслово с длиной периода $4n$, так что
$$W_{k-1}=u'_kx_{k'}^{4n}y'_k.$$
При этом положим
$$r_1 := \sup \{r: u'_k = u_k x_{k'}^r\},\quad r_2 :=\sup
\{r: y'_k = x_{k'}^r y_k\}.$$

(Отметим, что слова в наших рассуждениях могут быть пустыми.)\linebreak[3]
Определим  $f_k$ из равенства:
$$W_{k-1}= u_kx_{k'}^{4n+r_1+r_2}y_k = u_kf_ky_k.$$
Назов\"{е}м позиции, входящие в слово $f_k$, {\em скучными}, последнюю позицию слова $u_k$ --- {\em скучной типа $1$,} вторую с конца позицию $u_k$ --- {\em скучной типа $2$} и так далее, $n$-ую с конца позицию $u_k$ --- {\em скучной типа $n$}. Если позиция в процессе алгоритма определяется как скучная двух типов, то будем считать е\"{е} скучной того типа, который меньше. Положим $W_k = u_k y_k.$
\end{description}

\end{algorithm}

\begin{notation}
Провед\"{е}м $4t+1$ шагов алгоритма \ref{c:al}. Рассмотрим
первоначальное слово $W$. Для каждого натурального $i$ из отрезка
$[1,4t]$ имеет место равенство
$$
W = w_0f_i^{(1)}w_1f_i^{(2)}\cdots f_i^{(n_i)} w_{n_i}
$$
для некоторых подслов $w_j$. Здесь $f_i = f_i^{(1)}\cdots
f_i^{(n_i)}$. Также мы считаем, что при $1\leqslant j\leqslant
n_i-1$ подслово $w_j$ --- непустое. Пусть $s(k)$ --- количество
индексов $i\in[1,4t]$ таких, что $n_i = k$.
\end{notation}

Для доказательства теоремы \ref{c:main1} требуется найти как можно больше длинных периодических фрагментов. Помочь в этом сможет следующая лемма:

\begin{lemma}   \label{c:lem3}
$s = s(1) + s(2) \geqslant 2t$.
\end{lemma}

\begin{proof} Назов\"{е}м {\it монолитным}  подслово $U$ слова $W$, если
\begin{enumerate}
    \item $U$ является произведением слов вида $f_i^{(j)}$,
    \item $U$ не является подсловом слова, для которого выполняется предыдущее
свойство (1).
\end{enumerate}

Пусть после $(i-1)$-го шага алгоритма \ref{c:al} в слове $W$
содержится $k_{i-1}$ монолитных подслов. Заметим, что $k_i
\leqslant k_{i-1}-n_i+2$.

 Тогда если $n_i \geqslant 3$, то $k_i\leqslant k_{i-1}-1.$
Если же $n_i\leqslant 2$, то $k_i\leqslant k_{i-1}+1$. При этом
$k_1=1$, $k_t \geqslant 1 = k_1$. Лемма доказана. \end{proof}

\begin{corollary} \label{c:cor}
$$\sum\limits_{k=1}^\infty {k\cdot s(k)} \leqslant 10t\leqslant 5s.(\ref{c:cor})$$
\end{corollary}

\begin{proof} Из доказательства леммы \ref{c:lem3} получаем, что
$\sum\limits_{n_i\geqslant 3} {(n_i-2)} \leqslant 2t$.

По определению
$\sum\limits_{k=1}^\infty {s(k)} =4t$, т.е.
$\sum\limits_{k=1}^\infty {2s(k)} =8t$.

Складывая эти два неравенства и применяя лемму \ref{c:lem3}, получаем доказываемое
неравенство \ref{c:cor}.\end{proof}

\begin{proposition}
Высота слова $W$ будет не больше $$\Psi(n,4n,l)+\sum\limits_{k=1}^\infty {k\cdot s(k)}\leqslant \Psi(n,4n,l)+ 5s.$$
\end{proposition}

Далее будем рассматривать только $f_i$ с $n_i \leqslant 2$.

\begin{notation}
Если $n_i = 1$, то положим $f'_i := f_i$.

Ежели $n_i = 2$, то положим $f'_i := f_i^{(j)}$, где $f_i^{(j)}$
-- слово с наибольшей длиной между $f_i^{(1)}$ и $f_i^{(2)}$.

Слова $f'_i$ упорядочим в соответствии с их близостью к началу
$W$. Получим последовательность $f'_{m_1},\ldots ,f'_{m_s}$, где
$s'=s(1)+s(2)$, положим $f''_i := f'_{m_i}$. Пусть $f''_i = w'_i
x_{i''}^{p_{i''}}w''_i$, где хотя бы одно из слов $w'_i, w''_i$ --
пустое.
\end{notation}

\begin{remark}  \label{c:pr}
Можно считать, что мы первыми шагами алгоритма \ref{c:al}
выбрали все те $f_i$, для которых $n_i =1$.
\end{remark}

Теперь рассмотрим $z'_j$ --- подслова $W$
следующего вида:
$$z'_j = x_{(2j-1)''}^{p_{(2j-1)''}+\gimel}v_j,
\gimel\geqslant 0, |v_j| = |x_{(2j-1)''}|,$$

при этом $v_j$ не равно $x_{(2j - 1)''}$, начало подслова $z'_j$ совпадает
с началом периодического подслова в $f''_{2j-1}$. Покажем, что
$z'_j$ не пересекаются как подслова слова $W$.

В самом деле. Если $f''_{2j-1} = f_{m_{2j-1}}$, то
$z'_j=f_{m_{2j-1}}v_j$.

Если же $f''_{2j-1}= f_{m_{2j-1}}^{(k)}$, $k = 1,2$, а подслово
$z'_j$ пересекается с подсловом $z'_{j+1}$, то $f''_{2j}\subset z'_i$. Так
 как слова $x_{(2j)''}$ и $x_{(2j-1)''}$ --- нециклические, то
$|x_{(2j)''}| = |x_{(2j-1)''}|$. Но тогда длина периода в $z'_j$
не меньше $4n$, что противоречит замечанию \ref{c:pr}.

Тем самым доказана следующая лемма:
\begin{lemma}\label{lThick}
В слове $W$ с высотой не более $(\Psi(n,4n,l)+ 5s')$ найд\"{е}тся не менее $s'$ непересекающихся периодических подслов, в которых период повторится не менее $2n$ раз. Кроме того, между любыми двумя элементами данного множества периодических подслов найд\"{е}тся подслово длины периода более левого из выбранных элементов.
\end{lemma}

\subsection{Завершение доказательства субэкспоненциальности высоты}
Подставляя в лемму \ref{lThick} вместо числа $s'$ значение $s$ из доказательства теоремы \ref{ThThick} получаем, что
 высота $W$ не больше, чем
$$\Psi(n,4n,l)+ 5s < E_1 l\cdot n^{E_2+12\log_3 n} ,$$ где
$E_1 = 4^{21\log_3 4 + 17}, E_2 = 30\log_3 4 + 10.$

Тем самым мы получили {\bf утверждение
основной теоремы \ref{c:main2}}.


\newpage

\chapter{Оценки кусочной периодичности} \label{ch:pp}

\section{План улучшения оценок существенной высоты}

Далее приводятся оценки на количество периодических подслов с периодом длины $2, 3, (n-1)$ произвольного не $n$-разбиваемого слова $W$. Рассмотрение случая периодов длины $2, 3$ при помощи кодировки обобщается до доказательства ограниченности существенной высоты. Кроме того, получена нижняя оценка на число подслов с периодом $2,$ и эта оценка при достаточно большом $l$ отличается от верхней в 4 раза.

С целью дальнейшего улучшения оценок, полученных в главе \ref{ch:height}, вводятся следующие определения:

\begin{definition} \index{Высота!выборочная!малая}
а) Число $h$ называется {\em малой выборочной высотой} с границей $k$
слова $W$ над множеством слов $Z$, если $h$ --- такое максимальное
число, что у слова $W$ найд\"{е}тся $h$ попарно непересекающихся
циклически несравнимых подслов вида $z^m,$ где $z\in Z, m>k$.

б) Число $h$ называется {\em большой выборочной высотой} с границей $k$ \index{Высота!выборочная!большая}
слова $W$ над множеством слов $Z$, если $h$ --- такое максимальное
число,что у слова $W$ найд\"{е}тся $h$ попарно непересекающихся
подслов вида $z^m,$ где $z\in Z, m>k$, прич\"{е}м соседние подслова из
этой выборки несравнимы.

{\bf Здесь и далее:} $k = 2n$.

в) Множество слов $V$ имеет малую (большую) выборочную высоту $h$ \index{Высота!выборочная!множества}
над некоторым множеством слов $Z$, если $h$ является точной
верхней гранью малых (больших) вы\-бо\-роч\-ных высот над $Z$ его
элементов.
\end{definition}

Затем доказываются следующие нижние и верхние оценки на кусочную периодичность:

\begin{theorem}     \label{verh}
Малая выборочная высота множества не сильно $n$-разбиваемых слов
над $l$-буквенным алфавитом относительно множества нециклических
слов длины $2$ не больше $\beth(2,l,n)$, где
$$\beth(2,l,n) = {(2l-1)(n-1)(n-2)\over 2}.$$
\end{theorem}

\begin{theorem} \label{niz}
Малая выборочная высота множества не сильно $n$-разбиваемых слов
над $l$-буквенным алфавитом относительно множества нециклических
слов длины $2$ при фиксированном $n$ больше, чем $\alpha(n,l)$, где
$$\alpha(n,l) =  {n^2 l\over 2} (1 - o(l)).$$
Более точно, $\alpha(n,l) =(l-2^{n-1})(n-2)(n-3)/2.$
\end{theorem}

 \begin{theorem}     \label{verh1}
Малая выборочная высота множества не сильно $n$-разбиваемых слов
над $l$-буквенным алфавитом относительно множества нециклических
слов длины $3$ не больше $\beth(3,l,n)$, где
$$\beth(3,l,n) = {(2l-1)(n-1)(n-2)}.$$
\end{theorem}

Доказательства теорем \ref{niz}, \ref{verh}, \ref{verh1} изложены в работе \cite{Kh11}.

 \begin{theorem}     \label{verh2}
Малая выборочная высота множества не сильно $n$-разбиваемых слов
над $l$-буквенным алфавитом относительно множества нециклических
слов длины $(n-1)$ не больше $\beth(n-1,l,n)$, где
$$\beth(n-1,l,n) = (l-2)(n-1).$$
\end{theorem}

Теорема \ref{ess:t} с помощью кодировки обобщает теорему \ref{verh} до доказательства ограниченности существенной высоты множества не $n$-разбиваемых слов.

\begin{theorem} \label{ess:t}
Существенная высота $l$-порожд\"{е}нной $\PI$-алгебры с допустимым по\-ли\-но\-ми\-аль\-ным тождеством степени $n$ над множеством слов длины $<n$ меньше, чем $\Upsilon (n, l),$ где
$$\Upsilon (n, l) = 8(l+1)^n n^5(n-1).$$
\end{theorem}

Доказательства теорем \ref{verh2}, \ref{ess:t} изложены в работе \cite{Kh11(2)}.

Малую и большую выборочные высоты связывает следующая теорема:

\begin{theorem} \label{co1}
Большая выборочная высота $l$-порожд\"{е}нной $\PI$-алгебры $A$ с
до\-пус\-ти\-мым полиномиальным тождеством степени $n$ над множеством
нециклических слов длины $k$ меньше $2(n - 1)\beth(k,l,n)$, где $\beth(k,l,n)$ --- малая выборочная высота $A$ над множеством
нециклических слов длины $k$.
\end{theorem}

Как и ранее будем считать, что слова строятся над алфавитом $\gA$ из букв $\{a_1, a_2,\ldots,a_{l}\},$ над которыми введ\"{е}н лексикографический порядок, прич\"{е}м $a_i<a_j$, если $i<j$. Для следующих ниже доказательств будем отождествлять буквы $a_i$ с их индексами $i$ (то есть будем писать не слово $a_ia_j$, а слово $ij$).

\section{Доказательство верхних оценок выборочной высоты}
\subsection{Периоды длины два}

Пусть слово $W$ не сильно $n$-разбиваемо.  Рассмотрим некоторое
множество $\Omega''$ попарно непересекающихся циклических сравнимых подслов
$W$ вида $z^m$, где $m>2n$, $z$~---~нециклическое двухбуквенное
слово. Будем называть элементы этого множества {\it представителями,}
имея в виду, что эти элементы являются представителями раз\-лич\-ных
классов эквивалентности по сильной сравнимости. Пусть набралось
$t$ таких пред\-ста\-ви\-те\-лей. Пронумеруем их всех в порядке положения
в слове $W$ (первое --- самое левое) числами от $1$ до $t$. В
каждом выбранном представителе в качестве подслов содержатся ровно
два различных двухбуквенных слова.

Введ\"{е}м порядок на этих словах следующим образом: $u\prec v$, если
\begin{itemize}
    \item $u$ лексикографически меньше $v$,
    \item представитель, содержащий $u$
левее представителя, содержащего $v$.
\end{itemize}
Из не сильной
$n$-разбиваемости получаем, что максимальное возможное число
попарно несравнимых элементов равно $(n-1)$. По теореме Дилуорса
существует разбиение рассматриваемых двухбуквенных слов на $(n-1)$
цепь. Раскрасим слова в $(n-1)$ цвет в соответствии с их
принадлежностью к цепям.

Введ\"{е}м соответствие между следующими четырьмя
объектами:
\begin{itemize}
    \item натуральными числами от $1$ до $t$,
    \item классами эквивалентности по сильной сравнимости,
    \item содержащимися в классах эквивалентности по сильной сравнимости цик\-ли\-чес\-ки\-ми словами
длины $2$,
    \item парами цветов, в которые
раскрашены сдвиги по циклу этого слово-цикла.\index{Слово!-цикл}
\end{itemize}

Буквы слово-цикла раскрасим в цвета, в которые раскрашены сдвиги по циклу, начинающиеся с этих цветов.

Рассмотрим граф $\Gamma$ \index{Граф!подслов!длины два} с вершинами вида $(k,i)$, где $0<k<n$ и $0<i
\leqslant l$. Первая координата соответствует цвету, а вторая --- букве. Две вершины $(k_1,i_1), (k_2,i_2)$ со\-е\-ди\-ня\-ют\-ся
{\it ребром с весом  $j,$} если
\begin{itemize}
    \item в $j$-ом представителе содержится
слово-цикл из букв $i_1, i_2$,
    \item буквы $j$-го представителя раскрашены в
цвета $k_1, k_2$ соответственно.
\end{itemize}


Посчитаем число р\"{е}бер между вершинами вида $(k_1,i_1)$ и вершинами
вида $(k_2,i_2)$, где $k_1, k_2$ --- фиксированы, $i_1, i_2$ --
произвольны. Рассмотрим два ребра $l_1$ и $l_2$ из рассматриваемого множества с
весами $j_1 < j_2$ с концами в некоторых вершинах $A=(k_1, i_{1_1}),
B=(k_2,i_{2_1})$ и $C=(k_1,i_{1_2}), D=(k_2,i_{2_2})$ соответственно.
Тогда по построению одновременно выполняются неравенства $ i_{1_1}\leq i_{1_2}, i_{2_1}\leq  i_{2_2}.$
При этом, так как рассматриваются представители классов
эквивалентности по сильной срав\-ни\-мос\-ти, одно из неравенств
строгое. Значит, $i_{1_1}+i_{2_1} < i_{1_2}+i_{2_2}.$
Так как вторые координаты вершин ограничены числом $l$,
то вычисляемое число р\"{е}бер будет не более $(2l-1)$.

Так как первая координата вершин меньше $n$, то всего р\"{е}бер в
графе будет не более $(2l-1)(n-1)(n-2)/2$. Таким образом, теорема
\ref{verh} доказана.

\subsection{Периоды длины три}

Пусть слово $W$ не сильно
$n$-разбиваемо.  Рассмотрим некоторое множество попарно
непересекающихся циклических несравнимых подслов слова $W$ вида $z^m$,
где $m>2n$, $z$ --- нециклическое тр\"{е}хбуквенное слово. Будем называть
элементы этого множества представителями, имея в виду, что эти
элементы являются представителями различных классов
эквивалентности по сильной сравнимости. Пусть набралось $t$
таких пред\-ста\-ви\-те\-лей. Пронумеруем их всех в порядке положения в
слове $W$ (первое --- ближе всех к началу слова) числами от $1$ до $t$. В каждом
выбранном представителе в качестве подслов содержатся ровно три
различных тр\"{е}хбуквенных слова.

Введ\"{е}м порядок на этих словах следующим образом:

$u\prec v$, если
\begin{itemize}
    \item $u$ лексикографически меньше $v$,
    \item представитель, содержащий $u$,
левее представителя, содержащего $v$.
\end{itemize}
 Из не сильной
$n$-разбиваемости получаем, что максимальное возможное число
попарно несравнимых элементов равно $n-1$. По теореме Дилуорса
существует разбиение рассматриваемых тр\"{е}хбуквенных слов на $(n -
1)$ цепь. Раскрасим слова в $(n-1)$ цвет в соответствии с их
принадлежностью к цепям.

Можно заметить, что до этого момента доказательство теоремы \ref{verh1} практически полностью повторяет доказательство теоремы \ref{verh}. Однако для дальнейшего доказательства необходимо использовать ориентированный
аналог графа $\Gamma,$ который вводится далее.

Рассмотрим теперь уже ориентированный граф $G$ \index{Граф!подслов!длины три} с вершинами вида $(k,i)$, где
$0 < k < n$ и $0 < i \leqslant l$. Первая координата обозначает цвет, а вторая --- букву.
{\it Ребро с некоторым весом $j$}
выходит из $(k_1, i_1)$ в $(k_2, i_2)$, если для некоторых $i_3, k_3$
\begin{itemize}
    \item в $j$-ом представителе содержится слово-цикл $i_1 i_2 i_3$,
    \item буквы $i_1, i_2, i_3$ $j$-го представителя раскрашены в цвета $k_1, k_2, k_3$ соответственно.
\end{itemize}
Таким образом, граф $G$ состоит из ориентированных треугольников с
р\"{е}брами одинакового веса. Однако в отличие от графа $\Gamma$ из доказательства теоремы \ref{verh}, могут
появляться кратные р\"{е}бра, то есть р\"{е}бра с общими началом и концом, но разным весом. Для дальнейшего доказательства нам
потребуется следующая

\begin{lemma}[Основная] \label{ess:lB}
Пусть $A$, $B$ и $C$ --- вершины графа $G$, $A\to B\to C\to A$ --
ориентированный треугольник с р\"{е}брами некоторого веса $j$, кроме того,
су\-щест\-ву\-ют другие р\"{е}бра $A\to B, B\to C, C\to A$ с весами $a, b,
c$ соответственно. Тогда среди $a, b, c$ есть число, большее $j$.
\end{lemma}

\begin{proof} От противного. Если два числа из набора $a, b, c$ равны друг другу,
то  $a = b = c = j$, так как в противном случае есть 2 треугольника $A\to B\to C\to A$, в каждом из которых веса всех тр\"{е}х р\"{е}бер совпадают между собой. Тогда в $\Omega''$ есть два не сильно сравнимых слова, что противоречит определению $\Omega''.$ Без ограничения общности, что $a$ наибольшее из чисел
$a, b и с$. Рассмотрим треугольник из р\"{е}бер веса $a.$ Этот треугольник будет иметь общую с $\triangle ABC$ сторону $AB$ и некоторую третью вершину $C'.$ Если вторая координата вершины $C'$ совпадает со второй координатой вершины $C$ (то есть совпали соответствующие $C$ и $C'$ буквы алфавита), то $\triangle ABC$ и $\triangle ABC'$ соответствуют не сильно сравнимым словам из множества $\Omega''.$ Снова получено противоречие с определением множества $\Omega''.$ По предположению $a<j,$ а значит, из монотонности цвета $k_A$ (первой координаты вершины $A$) слово $i_A i_B i_{C'},$ составленное из вторых координат вершин $A, B, C'$ соответственно, лексикографически меньше слова $i_A i_B i_C.$
Значит, $i_{C'} < i_C.$ Тогда слово $i_B i_{C'}$ лексикографически меньше слова $i_B i_{C}.$ Из монотонности цвета $k_B$ получаем, что $b>a.$ Противоречие.
\end{proof}
\medskip

\subsection{Завершение доказательства теоремы \ref{verh1} }

Рассмотрим теперь граф $G_1$, полученный из графа $G$ заменой
между каждыми двумя вершинами кратных р\"{е}бер на ребро с наименьшим
весом. Тогда по лемме \ref{ess:lB} в графе $G_1$ встретятся р\"{е}бра всех весов от 1 до
$t$.

Посчитаем число р\"{е}бер из вершин вида $(k_1, i_1)$ в вершины вида
$(k_2, i_2)$, где $k_1, k_2$ фиксированы, $i_1, i_2$
произвольны. Рассмотрим два ребра из рассматриваемого мно\-жест\-ва с
весами $j_1 < j_2$ с концами в некоторых вершинах $(k_1, i_{1_1}),
(k_2, i_{2_1})$ и $(k_1,i_{1_2}),(k_2,i_{2_2})$ соответственно.
Тогда по построению $ i_{1_1}\leqslant i_{1_2}, i_{2_1}\leqslant i_{2_2}$,
прич\"{е}м, так как рас\-смат\-ри\-ва\-ют\-ся представители классов
эквивалентности по сильной сравнимости, одно из неравенств
строгое. Так как вторые координаты вершин ограничены числом $l$,
то вычисляемое число р\"{е}бер будет не более $2l-1$.

Так как первая координата вершин меньше $n$, то всего р\"{е}бер в
графе будет не более $(2l-1)(n-1)(n-2)$. Таким образом, теорема
\ref{verh1} доказана.

\subsection{Периоды длины, близкой к степени тождества в алгебре} \label{ss:big_periods}

Пусть слово $W$ не $n$-разбиваемо. Как и прежде, рассмотрим некоторое множество попарно непересекающихся несравнимых подслов слова $W$ вида $z^m,$ где $m>2n,$ $z$ --- $(n-1)$-буквенное нециклическое слово.  Будем называть
элементы этого множества {\it представителями}, имея в виду, что эти
элементы являются представителями раз\-лич\-ных классов
эквивалентности по сильной сравнимости. Пусть набралось $t$
таких предс\-та\-ви\-те\-лей. Пронумеруем их всех в порядке положения в
слове $W$ (первое --- ближе всех к началу слова) числами от $1$ до $t$. В каждом выбранном представителе в качестве подслов содержатся ровно $(n-1)$ различное $(n-1)$-буквенное слово.

Введ\"{е}м порядок на этих словах следующим образом: $u\prec v$, если
\begin{itemize}
    \item $u$ лексикографически меньше $v$;
    \item представитель, содержащий $u$ левее представителя, содержащего $v$.
\end{itemize}
Из не сильной
$n$-разбиваемости получаем, что максимальное возможное число
попарно несравнимых элементов равно $n-1$. По теореме Дилуорса существует разбиение рассматриваемых $(n-1)$-буквенных слов на $(n-1)$ цепь. Раскрасим слова в $(n-1)$ цвет в соответствии с их принадлежностью к цепям. Раскрасим позиции, с которых начинаются слова, в те же цвета, что и соответствующие слова.

Рассмотрим ориентированный граф $G$ с вершинами вида $(k,i)$, где \index{Граф!подслов!большой длины}
$0 < k < n$ и $0 < i \leqslant l$. Первая координата обозначает цвет, а вторая --- букву.

Ребро с некоторым весом $j$
выходит из $(k_1, i_1)$ в $(k_2, i_2)$, если
\begin{itemize}
    \item для некоторых $i_3,i_4,\ldots,i_{n-1}$ в $j$-ом представителе содержится слово-цикл $i_1i_2\cdots i_{n-1}$;
    \item позиции, на которых стоят буквы $i_1, i_2$ раскрашены в цвета $k_1, k_2$ соответственно.
\end{itemize}
Таким образом, граф $G$ состоит из ориентированных циклов длины $(n-1)$ с р\"{е}брами одинакового веса. Теперь нам требуется найти показатель, который бы строго монотонно рос с появлением каждого нового представителя при движении от начала к концу слова $W.$ В теореме \ref{verh1} таким показателем было число несократимых р\"{е}бер графа $G.$ В доказательстве теоремы \ref{verh2} будет рассматриваться сумма вторых координат неизолированных вершин графа $G.$ Нам потребуется следующая

\begin{lemma}[Основная]\label{n-1:l}
Пусть $A_1, A_2,\ldots,A_{n-1}$ --- вершины графа $G,$ $A_1\to A_2\to\cdots\to A_{n-1}\to A_1$ --- ориентированный цикл длины $(n-1)$ с р\"{е}брами некоторого веса $j.$ Тогда не найд\"{е}тся другого цикла между вершинами $A_1, A_2,\ldots,A_{n-1}$ одного веса.
\end{lemma}
 \begin{proof}От противного. Рассмотрим наименьшее число $j$, для которого наш\"{е}лся другой одноцветный цикл между вершинами цикла цвета $j.$ В силу минимальности $j$ можно считать, что этот цикл имеет цвет $k>j.$ Пусть цикл цвета $k$ имеет вид $A_{j_1}, A_{j_2},\ldots,A_{j_{n-1}},$ где ${\{j_p\}}_{p=1}^{n-1} = \{1, 2,\ldots,n-1\}.$ Пусть $(k_j, i_j)$ --- координата вершины $A_j.$ Рассмотрим наименьшее число $q\in \mathbb N$ такое, что для некоторого $r$ слово
$i_{j_r}i_{j_{r+1}}\cdots i_{j_{r+q-1}}$ лек\-си\-ко\-гра\-фи\-чес\-ки больше слова $i_{j_r}i_{j_r+1}\cdots i_{j_r+q-1}$ (здесь и далее сложение нижних индексов происходит по модулю $(n-1)$). Такое $q$ существует, так как слова $i_1i_2\cdots i_{n-1}$ и $i_{j_1}i_{j_2}\cdots i_{j_{n-1}}$  сильно сравнимы. Кроме того, в силу совпадения множеств  ${\{j_p\}}_{p=1}^{n-1}$ и $ \{1, 2,\ldots,n-1\}$ получаем, что $q\geqslant 2.$ Так как $q$ --- наименьшее, то для любого $ s<q,$ любого $r$ имеем $ i_{j_r}i_{j_{r+1}}\cdots i_{j_{r+s-1}} = i_{j_r}i_{j_r+1}\cdots i_{j_r+s-1}.$ Тогда для любого $ s<q,$ любого $r$ имеем $ i_{j_{r+s-1}}= i_{j_r+s-1}.$ Из монотонности слов каждого цвета получаем, что для любого $r$ $ i_{j_r}i_{j_{r+1}}\cdots i_{j_{r+q-1}}$ не больше $i_{j_r}i_{j_r+1}\cdots i_{j_r+q-1}.$ Значит, для любого $r$ верно неравенство $i_{j_{r+q-1}}\geqslant i_{j_r+q-1}.$ По предположению найд\"{е}тся такое $r$, что $i_{j_{r+q-1}}> i_{j_r+q-1}.$ Так как обе последовательности ${\{j_{r+q-1}\}}_{к=1}^{n-1}$ и  ${\{j_r+q-1\}}_{к=1}^{n-1}$ пробегают элементы множества чисел $\{1, 2,\ldots,n-1\}$ по одному разу, то $\sum\limits_{r=1}^{n-1}j_{r+q-1} = \sum\limits_{r=1}^{n-1}(j_r+q-1)$ (при вычислении числа $j_r+q-1$ суммирование также проходит по модулю $(n-1)$). Но мы получили $\sum\limits_{r=1}^{n-1}j_{r+q-1} > \sum\limits_{r=1}^{n-1}(j_r+q-1).$ Противоречие. \end{proof}
\medskip

\subsection{Завершение доказательства теоремы \ref{verh2}}

Для произвольного $j$ рассмотрим циклы длины $(n-1)$ весов $j$ и $j+1$ для некоторого $j.$ Из основной леммы \ref{n-1:l} найдутся числа $k, i$ такие, что вершина $(k,i)$ входит в цикл веса $(j+1),$ но не входит в цикл веса $j.$ Пусть цикл веса $j$ состоит из вершин вида $(k, i_{(j,k)}),$ где $k=1,2,\ldots,n-1.$ Введ\"{е}м величину $\pi(j) = \sum\limits_{k=1}^{n-1}i_{(j,k)}.$ Тогда из основной леммы \ref{n-1:l} и монотонности слов по цветам получаем, что $\pi(j+1)\geqslant\pi(j)+1.$ Так как рассматриваемые периоды не циклические, то найд\"{е}тся $k$ такое, что $i_{(1,k)}>1.$ Значит, $\pi(1)>n-1.$ $\forall j:i_{(j,k)}\leqslant l-1,$ а значит, $\pi(j)\leqslant (l-1)(n-1).$ Следовательно, $j\leqslant (l-2)(n-1).$ Значит, $t\leqslant (l-2)(n-1).$ Тем самым, теорема \ref{verh2} доказана.

Представленная при доказательстве теоремы \ref{verh2} техника позволяет доказать следующий факт:

\begin{proposition}
Малая выборочная высота множества не сильно $n$-разбиваемых слов
над $l$-буквенным алфавитом относительно множества нециклических
слов длины $(n-c)$ не больше $D(c) n^c l,$ где $D(c)$ --- некоторая функция, зависящая от $c$.
\end{proposition}

\section{Нижняя оценка малой выборочной высоты над периодами длины два}

Привед\"{е}м пример. Из
формулировки этой теоремы следует, что можно положить $l$ сколь
угодно большим. Будем считать, что $l>2^{n-1}$. Мы воспользуемся кон\-струк\-ци\-я\-ми, принятыми в
доказательстве теоремы \ref{verh}. Таким образом, процесс
построения примера сводится к построению р\"{е}бер в графе на $l$
вершинах. Разобь\"{е}м этот процесс на несколько больших шагов. Пусть $i$ --- натуральное число от 1 до $(l - 2^{n - 1})$.
Пусть
на $i$-ом большом шаге в привед\"{е}нном ниже порядке соединяются
р\"{е}брами следующие пары вершин:\\
$(i,2^{n-2}+i)$,\\
$(i,2^{n-2}+2^{n-3}+i),(2^{n-2}+i,2^{n-2}+2^{n-3}+i)$,\\
$(i,2^{n-2}+2^{n-3}+2^{n-4}+i),(2^{n-2}+i,2^{n-2}+2^{n-3}+2^{n-4}+i),\\
(2^{n-2}+2^{n-3}+i,2^{n-2}+2^{n-3}+2^{n-4}+i),\ldots,$\\
$(i,2^{n-2}+\ldots+2+1+i),\ldots,(2^{n-2}+\ldots+2+i,2^{n-2}+\ldots+2+1+i), $\\
 где $i=2,3,\ldots, l-2^{n-1}+1.$\\
 При этом:
\begin{enumerate}
\item Никакое ребро не будет подсчитано 2 раза, так как вершина
соединена р\"{е}брами только с вершинами, значения в которых
отличаются от значения в выбранной вершине на неповторяющуюся
сумму степеней двойки.

\item Пусть {\it вершина типа} $(k,i)$ --- вершина, которая на $i$-ом
шаге соединяется с $k$ вершинами, значения в которых меньше
значения е\"{е} самой. Для всех $i$ будут вершины типов $(0,i),
(1,i)\ldots,(n-2,i)$.

Раскрасим в $k$-ый цвет слова, которые для некоторого $i$
начинаются с буквы типа $(k,i)$ и заканчиваются в буквах, с
которыми вершина типа $(k, i)$ со\-е\-ди\-ня\-ет\-ся р\"{е}брами на $i$-ом
большом шаге. Получена корректная раскраска в $(n-1)$ цвет, а
значит, слово сильно $n$-разбиваемо.

\item На $i$-ом большом шаге осуществляется $(n-2)(n-3)\over 2$
шагов. Значит, $$q=(l-2^{n-1})(n-2)(n-3)/2,$$ где $q$ --- количество р\"{е}бер в графе $\Gamma.$
\end{enumerate}
Тем самым, теорема \ref{niz} доказана.

\section{Оценка существенной высоты с помощью выборочной высоты} \label{s:cod}
Из рассмотрения случая периодов длины 2 с помощью кодировки букв можно получить оценку на существенную высоту, которая будет расти полиномиально по числу об\-ра\-зу\-ю\-щих и экспоненциально по степени тождества. Для этого надо обобщить некоторые понятия, введ\"{е}нные ранее. Заметим, что механизм кодировки букв представляется перспективным для обобщения оценок на высоту, полученных при конкретном значении одного из параметров (в данном случае --- ограничение длины слов в базисе Ширшова).

\begin{construction}
Рассмотрим алфавит $\gA$ с буквами $\{a_1, a_2,\ldots, a_l\}.$ Введ\"{е}м на буквах лексикографический порядок: $a_i>a_j,$ если $i>j.$ Рассмотрим произвольное множество нециклических попарно сильно сравнимых слово-циклов некоторой оди\-на\-ко\-вой длины $t.$ Пронумеруем элементы этого множества натуральными числами, начиная с 1. Введ\"{е}м порядок на словах, входящих в слово-цикл, следующим образом:
$u\prec v,$ если:
\begin{enumerate}
    \item слово $u$ --- лексикографически меньше слова $v,$
    \item слово-цикл, содержащий слово $u,$ имеет меньший номер, чем слово-цикл, со\-дер\-жа\-щий слово $v.$
\end{enumerate}
Пронумеруем теперь позиции букв в слово-циклах числами от 1 до $t$ от начала к концу некоторого слова, входящего в слово-цикл.
\begin{notation}
\begin{enumerate}
    \item Пусть $w(i, j)$ --- слово длины $t$, которое начинается с $j$-ой буквы в $i$-ом слово-цикле.
    \item Пусть класс $X(t, l)$ --- рассматриваемое множество слово-циклов с введ\"{е}нным на его словах порядком $\prec.$
\end{enumerate}
\end{notation}

\end{construction}

\begin{definition}
Назов\"{е}м те классы $X,$ в которых не найд\"{е}тся антицепи длины $n$, --- {\em $n$-светлыми.} Соответственно, те, в которых найд\"{е}тся такая антицепь --- {\em $n$-т\"{е}мными.}
\end{definition}

Из теоремы Дилуорса получаем, что слова в $n$-светлых классах $X$ можно рас\-кра\-сить в $(n-1)$ цвет, так что одноцветные слова образуют цепь. Далее требуется оценить число элементов в $n$-светлых классах $X$.

\begin{definition}
Пусть $\beth(t, l, n)$ --- наибольшее возможное число элементов в $n$-светлом классе $X(t, l).$
\end{definition}
\begin{remark}\label{ess:r}
Здесь и далее первый аргумент в функции $\beth(\cdot, \cdot, \cdot)$ меньше третьего.
\end{remark}
Следующая лемма позволяет оценить $\beth(t, l, n)$ через случаи малых периодов.
\begin{lemma}\label{ess:l1}
$\beth(t, l^2, n)\geqslant\beth(2t, l, n)$
\end{lemma}
 \begin{proof}Рассмотрим $n$-светлый класс $X(2t, l).$ Разобь\"{е}м во всех его слово-циклах позиции на пары соседних так, чтобы каждая позиция попала ровно в одну пару. Затем рассмотрим алфавит $\gB$ с буквами $\{b_{i,j}\}_{i,j=1}^l,$ прич\"{е}м $b_{i_1, j_1}>b_{i_2,j_2}, $ если $i_1\cdot l+j_1>i_2\cdot l+j_2.$ Алфавит $\gB$ состоит из $l^2$ букв. Каждая пара позиций из разбиения состоит из некоторых букв $a_i, a_j$. Заменим пару букв $a_i, a_j$ буквой $b_{i,j}.$ Поступая так с каждой парой, получаем новый класс $X(t, l^2).$ Он будет $n$-светлым, так как если в классе $X(t, l^2)$ есть антицепь длины $n$ из слов $w(i_1, j_1), w(i_2, j_2),\ldots,w(i_n, j_n),$ то следует рассматривать прообразы слов $w(i_1, j_1), w(i_2, j_2),\ldots,w(i_n, j_n)$ в первоначально взятом классе $X(2t, l).$ Пусть эти прообразы --- слова $w(i_1, j'_1), w(i_2, j'_2),\ldots,w(i_n, j'_n)$. Тогда слова $w(i_1, j'_1), w(i_2, j'_2),\ldots,w(i_n, j'_n)$ образуют в классе $X(2t, l)$ антицепь длины $n$. Получено противоречие с тем, что класс $X(2t, l)$ --- $n$-светлый. Тем самым, лемма доказана.\end{proof}

Теперь оценим $\beth(t, l, n)$ через случаи малых алфавитов.

\begin{lemma}
$\beth(t, l^2, n)\leqslant\beth(2t, l, 2n-1)$
\end{lemma}

 \begin{proof}Рассмотрим $(2n-1)$-т\"{е}мный класс $X(2t, l).$ Можно считать, что $n$ слов из антицепи, а именно $w(i_1, j_1), w(i_2, j_2),\ldots,w(i_n, j_n)$, начинаются с неч\"{е}тных позиций слово-циклов. Разобь\"{е}м во всех его слово-циклах позиции на пары соседних так, чтобы каждая позиция попала ровно в одну пару и первая позиция в каждой паре была неч\"{е}тной. Затем рассмотрим алфавит $\gB$ с буквами $\{b_{i,j}\}_{i,j=1}^l,$ прич\"{е}м\linebreak $b_{i_1, j_1}>b_{i_2,j_2}, $ если $i_1\cdot l+j_1>i_2\cdot l+j_2.$ Алфавит $\gB$ состоит из $l^2$ букв. Каждая пара позиций из разбиения состоит из некоторых букв $a_i, a_j$. Заменим пару букв $a_i, a_j$ буквой $b_{i,j}.$ Поступая так с каждой парой, получаем новый класс $X(t, l^2).$ Пусть слова $w(i_1, j_1), w(i_2, j_2),\ldots,w(i_n, j_n)$ перешли в слова $w(i_1, j'_1), w(i_2, j'_2),\ldots,w(i_n, j'_n).$ Эти слова будут образовывать антицепь длины $n$ в классе $X(t, l^2).$ Таким образом, получен $n$-т\"{е}мный класс $X(t, l^2)$ с тем же числом элементов, что и $(2n-1)$-т\"{е}мный класс $X(2t, l).$ Тем самым, лемма доказана.\end{proof}

Для дальнейшего рассуждения необходимо связать $\beth(t, l, n)$ для произвольного первого аргумента и для первого аргумента, равного степени двойки.

\begin{lemma}\label{ess:l3}
$\beth(t, l, n)\leqslant\beth(2^s, l+1, 2^s (n-1)+1),$ где $s = \ulcorner\log_2(t)\urcorner.$
\end{lemma}

 \begin{proof}Рассмотрим $n$-светлый класс $X(t,l).$ Введ\"{е}м в алфавит $\gA$ новую букву $a_0$, которая лексикографически меньше любой другой буквы из алфавита $\gA.$ Получен алфавит $\gA'.$ В каждый слово-цикл из класса  $X(t,l)$ добавим $(t+1)$-ю, $(t+2)$-ю,$\ldots,$\linebreak $2^s-$ю позиции, на которые поставим буквы $a_0.$ Получили класс $X(2^s, l+1).$ Он будет $(2^s (n-1)+1)$-светлым, так как в противном случае в этом классе для некоторого $j$ найдутся слова  $w(i_1, j), w(i_2, j),\ldots,w(i_n, j),$ которые образуют антицепь в классе $X(2^s, l+1).$ Тогда:
\begin{enumerate}
	\item Если $j>t,$ то слова  $w(i_1, 1), w(i_2, 1),\ldots,w(i_n, 1)$ образуют антицепь в классе  $X(t,l).$
	\item Если $j\leqslant t,$ то слова  $w(i_1, j), w(i_2, j),\ldots,w(i_n, j)$ образуют антицепь в классе  $X(t,l).$
\end{enumerate}
Получено противоречие с тем, что класс $X(t,l)$ --- $n$-светлый. Тем самым, лемма доказана.\end{proof}

\begin{proposition}
$\beth(t, l, n)\leqslant\beth(t, l, n+1)$
\end{proposition}

По лемме \ref{ess:l3} $\beth(t, l, n)\leqslant\beth(2^s, l+1, 2^s (n-1)+1),$ где $s = \ulcorner\log_2(t)\urcorner.$

В силу замечания \ref{ess:r} $t<n$. Значит, $2^s<2n.$

Следовательно, $\beth(2^s, l+1, 2^s (n-1)+1)\leqslant\beth(2^s, l+1, 2n^2).$

По лемме \ref{ess:l1} имеем $$\beth(2^s, l+1, 2n^2)\leqslant\beth(2^{s-1},( l+1)^2, 2n^2)\leqslant\beth(2^{s-2}, (l+1)^{2^2}, 2n^2)\leqslant$$$$\leqslant\beth(2^{s-3}, (l+1)^{2^3}, 2n^2)\leqslant\cdots\leqslant\beth(2, (l+1)^{2^{s-1}},2n^2).$$

По теореме \ref{verh}  имеем $\beth(2, (l+1)^{2^{s-1}},2n^2)<(l+1)^{2^{s-1}}\cdot 4n^4< 4(l+1)^n n^4.$

То есть доказана следующая
\begin{lemma}\label{ess:l4}
$\beth(t, l, n)<4(l+1)^n n^4.$
\end{lemma}

Чтобы применить лемму \ref{ess:l4} к доказательству теоремы \ref{ess:t}, требуется оценить число подслов не $n$-разбиваемого слова с одинаковыми периодами.

\begin{lemma}\label{ess:l5}
Если в некотором слове $W$ найдутся $(2n-1)$ подслов, в которых период повторится больше $n$ раз, и их периоды попарно не сильно сравнимы, то $W$ --- $n$-разбиваемое.
\end{lemma}
 \begin{proof}Пусть в некотором слове $W$ найдутся $(2n-1)$ подслов, в которых период повторится больше $n$ раз, и их периоды попарно не сильно сравнимы. Пусть $x$ --- период одного из этих подслов. Тогда в слове $W$ найдутся непересекающиеся подслова $x^{p_1}v'_1,\ldots ,\linebreak x^{p_{2n-1}}v'_{2n-1},$ где $p_1,\ldots ,p_{2n-1}$ --- некоторые натуральные числа, большие $n,$ а $v'_1,\ldots ,v'_{2n-1}$ --- некоторые слова длины $|x|$, сравнимые с $x.$ Тогда среди слов $v'_1,\ldots ,v'_{2n-1}$ найдутся либо $n$ лексикографически больших $x$, либо $n$ лексикографически меньших $x$. Можно считать, что $v'_1,\ldots ,v'_n$ --- лексикографически больше $x$. Тогда в слове $W$ найдутся подслова $v'_1, xv'_2,\ldots ,x^{n-1}v'_n,$ идущие слева направо в порядке лексикографического убывания.\end{proof}

Из этой леммы получаем следствие \ref{co1}.

Рассмотрим не $n$-разбиваемое слово $W.$ Если в н\"{е}м найд\"{е}тся подслово, в котором нециклический период $x$ длины не меньше $n$ повторится больше $2n$ раз, то в слове $x^2$ подслова, которые начинаются с первой, второй$,\ldots,n$-ой позиции, попарно сравнимы. Значит, слово $x^{2n}$ является $n$-разбиваемым. Получаем противоречие с не $n$-разбиваемостью слова $W.$  Из лемм \ref{ess:l5} и \ref{ess:l4} получаем, что существенная высота слова $W$ меньше, чем $(2n-1)\sum\limits_{t=1}^{n-1} \beth(t, l, n) < 8(l+1)^n n^5 (n-1).$ Значит, $\Upsilon(n, l)<8(l+1)^n n^5 (n-1).$ Тем самым, теорема \ref{ess:t} доказана.

\newpage

\chapter{Оценки числа перестановочно упорядоченных множеств} \label{ch:pu_number}

\section{Введение и основные понятия}

\begin{definition} \index{Множество!перестановочно упорядоченное}
Частично упорядоченное множество $M$ называется {\em перестановочно упорядоченным,} если порядок на н\"{е}м есть пересечение двух линейных порядков.
\end{definition}

Рассмотрим теперь некоторую перестановку $\pi$ элементов $1, 2,\dots , n$ (иначе говоря, $\pi\in S_n$). Определим понятие $k$-разбиваемости.

\begin{definition}
Пусть для перестановки $\pi\in S_n$ найд\"{е}тся последовательность натуральных чисел $1\leqslant i_1\leqslant i_2\leqslant\dots\leqslant i_k$ таких, что $\pi(i_1)\geqslant\pi(i_2)\geqslant\dots\geqslant\pi(i_k).$ Тогда перестановка $\pi(1)\pi(2)\dots\pi(n)$ называется $k$-разбиваемой.
\end{definition}

\begin{example}
Количество не 3-разбиваемых перестановок в группе $S_n$ есть $n$-е число Каталана и равно ${(2n)!\over n!(n+1)!}$.
\end{example}

\begin{proposition}
Если слово является $k$-разбиваемым, то для любого $m<k$ оно также является $m$-разбиваемым.
\end{proposition}

Далее нам потребуется определение диаграммы Юнга.

\begin{definition} \index{Диаграмма Юнга!стандартная}
{\em (Стандартной) диаграммой Юнга порядка $n$} называется таблица, в ячейках которой написаны $n$ различных натуральных чисел, прич\"{е}м суммы чисел в каждой строке и каждом столбце возрастают, между числами нет пустых ячеек и есть элемент, который содержится и в первой строке, и в первом столбце.
\end{definition}

\begin{definition}
Диаграмма Юнга называется {\em диаграммой формы} $p = (p_1, p_2,\dots , p_m),$ если у не\"{е} $m$ строк и $i$-я строка имеет длину $p_i.$
\end{definition}

Формы диаграмм Юнга пробегают все возможные разбиения на циклы элементов симметрической группы $S_n$. Любой класс сопряж\"{е}нности группы $S_n$ зада\"{е}тся некоторым разбиением на циклы. Каждому классу сопряж\"{е}нности группы соответствует некоторое е\"{е} неприводимое представление. Следовательно, форма диаграммы Юнга соответствует неприводимому представлению группы $S_n$.

Пронумеруем все клетки диаграммы Юнга формы $p$ числами от 1 до $n$. Пусть $h_k$ --- количество клеток диаграммы Юнга, расположенных
\begin{itemize}
    \item либо в одной строке, либо в одном столбце с клеткой с номером $k$,
    \item находящихся не левее или не выше клетки с номером $k$.
\end{itemize}
Тогда число диаграмм Юнга формы $p$ и равная ему размерность соответствующего неприводимого представления группы $S_n$, вычисляются по ``формуле крюков'' ${n!\over\prod\limits_{k=1}^n h_k}$.

В работе \cite{Sch61} приведена биекция между перестановками $\pi$ чисел $1, 2,\dots , n$ и заполненных теми же числами парами диаграмм Юнга $(P, Q)$. Эта биекция и е\"{е} следствия будут разобраны в главе \ref{sec:main}.

{\it В нашей работе мы доказываем следующие результаты:}

\begin{theorem}\label{th:main}
$\xi_k(n)$ --- количество не $(k+1)$-разбиваемых перестановок $\pi\in S_n$ --- не больше, чем ${k^{2n}\over ((k-1)!)^2}$.
\end{theorem}

\begin{theorem}\label{th:main2}
$\varepsilon_k(n)$ --- количество $n$-элементных перестановочно упорядоченных множеств с максимальной антицепью длины $k$ --- не больше, чем $\min\{{k^{2n}\over (k!)^2}, {(n-k+1)^{2n}\over ((n-k)!)^2}\}$.
\end{theorem}

\begin{corollary}\label{cor:main}
Пусть $\digamma$ является множеством слов алфавита из $l$ букв с введ\"{е}нным на них лексикографическим порядком. Назов\"{е}м {\em полилинейным} слово, все буквы которого различны.  Назов\"{е}м слово {\em $k$-разбиваемым,} если в н\"{е}м найдутся $k$ непересекающихся подслов, идущих в порядке лексикографического убывания. Тогда количество полилинейных слов длины $n$ $(n\leqslant l)$, не являющихся $(k+1)$-разбиваемыми, не больше, чем ${l!k^{2n}\over n!(l-n)!((k-1)!)^2}$.
\end{corollary}

Оценка в теореме \ref{th:main} улучшает полученную в работе \cite{Lat72}. Следует сказать, что оценка на $\xi_k(n)$ в работе \cite{Lat72} была получена для доказательства теоремы Регева, вопрос же о е\"{е} точности не ставился. Оценка в работе \cite{Lat72} доказывается с помощью теоремы Дилуорса. \index{Теорема!Дилуорса} Применение теоремы Дилуорса в некоторых других задачах комбинаторики слов описано в работе \cite{BK12}.

В работе \cite{Ch07} доказывается, что для для определ\"{е}нной функции $K(n) = o(\sqrt[3]{n}\ln n)$ и числа  $k \leqslant K(n) = o(\sqrt[3]{n}\ln n)$ верна асимптотическая оценка $\xi_k(n) = k^{2n - o(n)}$.

Для получения производящей функции в работе \cite{Kn70} введено следующее понятие:

\begin{definition} \index{Диаграмма Юнга!обобщ\"{е}нная}
Обобщ\"{е}нной диаграммой Юнга формы $(p_1, p_2,\dots  p_m)$, где $p_1\geqslant p_2\geqslant\dots\geqslant p_m\geqslant 1$, называется массив $Y$ положительных чисел $y_{ij}$, где $1\leqslant j\leqslant p_i$, $1\leqslant i\leqslant m,$ такой, что числа в его строках не убывают, а в столбцах возрастают.
\end{definition}

Ещ\"{е} требуются двухстрочные массивы следующего типа.

\begin{definition}
Набор пар положительных чисел $(u_1, v_1),$ $(u_2, v_2),\dots,$ $(u_N, v_N)$ такой, что пары $(u_k, v_k)$ расположены в неубывающем лексикографическом порядке, называется {\em набором типа $\alpha(N)$}.
\end{definition}
В работе \cite{Kn70} устанавливается биекция между наборами типа $\alpha(N)$ и парами $(P, Q)$ обобщ\"{е}нных диаграмм Юнга порядка $N$ (т. е. состоящих из $N$ ячеек). Кроме того, существует взаимно-однозначное соответствие между рассматриваемыми наборами и матрицами, в которых число в ячейке из $i$-ой строки и $j$-го столбца равно количеству пар $(i, j)$ в наборе.
В работе \cite{Ges90} на основании функций Шура $s_\lambda$, которые также являются производящими функциями для обобщ\"{е}нных диаграмм Юнга, строится производящая функция для $\xi_k(n).$ Однако сложность построения явной формулы для $\xi_k(n)$ раст\"{е}т экспоненциально по $k$. К примеру, $\xi_3(n)= 2 \sum\limits_{ k=0}^n \bigl(\begin{smallmatrix}2k\\ k \end{smallmatrix}\bigr)\bigl(\begin{smallmatrix}n\\ k \end{smallmatrix}\bigr)^2 {3k^2 + 2k + 1 - n - 2kn\over (k + 1)^2(k + 2)(n - k + 1)}$.

\section{Алгебраические обобщения}

В 1950 году Шпехт (\cite{Sp50}) поставил проблему существования бесконечно базируемого многообразия ассоциативных алгебр над полем характеристики 0. Решение проблемы Шпехта для нематричного случая представлено в докторской диссертации В. Н. Латышева \cite{Lat77}. Рассуждения В. Н. Латышева основывались на применении техники частично упорядоченных множеств. А. Р. Кемер (\cite{Kem87}) доказал, что каждое многообразие ассоциативных алгебр конечно базируемо, тем самым решив проблему Шпехта.

Первые примеры бесконечно базируемых ассоциативных колец были получены А. Я. Беловым (\cite{Bel99}), А. В. Гришиным (\cite{Gr99}) и В. В. Щиголевым (\cite{Shch99}).

После решения проблемы Шпехта в случае характеристики 0 актуален вопрос, поставленный Латышевым.

Введ\"{е}м некоторый порядок на словах алгебры над полем. Назов\"{е}м {\it обструкцией} полилинейное слово, которое
\begin{itemize}
    \item является уменьшаемым (т. е. является комбинацией меньших слов);
    \item не имеет уменьшаемых подслов;
    \item не является изотонным образом уменьшаемого слова меньшей длины.
\end{itemize}

\begin{ques}[Латышев]
Верно ли, что количество обструкций для полилинейного $T$-идеала конечно?
\end{ques}

Из проблемы Латышева вытекает полилинейный случай проблемы конечной базируемости для алгебр над полем конечной характеристики. Наиболее важной обструкцией является обструкция $x_n x_{ n-1}\dots x_1$, е\"{е} изотонные образы составляют множество не $n$-разбиваемых слов.

В связи с этими вопросами возникает проблема:

\begin{ques}
Перечислить количество полилинейных слов, отвечающих данному конечному набору обструкций. Доказать элементарность соответствующей производящей функции.
\end{ques}

\section{Доказательство основных результатов}\label{sec:main}

\begin{lemma}[\cite{Sch61}]\label{lem:schlem3}
Существует взаимооднозначное соответствие между перестановками $\pi\in S_n$ и парами $(P, Q)$ стандартных диаграмм Юнга, заполненных числами от 1 до $n$ и такими, что форма $P$ совпадает с формой~$Q$.
\end{lemma}

\begin{proof}
Пусть $\pi = x_1 x_2\dots x_n$. Построим по ней пару диаграмм Юнга $(P, Q)$. Сначала построим диаграмму $P$.

Определим операцию $S\leftarrow x$, где $S$ --- диаграмма Юнга, $x$ --- натуральное число, не равное ни одному из чисел в диаграмме $S$.
\begin{enumerate}
\item Если $x$ не меньше самого правого числа в первой строке $S$ (если в ней нет чисел, то будем считать, что $x$ больше любого из них), то добавляем $x$ в конец первой строки диаграммы $S$. Полученная диаграмма $S\leftarrow x$.
\item Если найд\"{е}тся большее, чем $x$, число в первой строке $S$, то пусть $y$ --- наименьшее число в первой строке, такое что $y > x$. Тогда заменим $y$ на $x$. Далее проводим с $y$ и второй строкой те же действия, что проводили с $x$ и первой строкой.
\item Продолжаем этот процесс строка за строкой, пока какое-нибудь число не будет добавлено в конец строки.
\end{enumerate}
Из построения $S\leftarrow x$ получаем, что вновь полученная таблица будет диаграммой Юнга.

Пусть $P = (\dots ((x_1\leftarrow x_2)\leftarrow x_3)\dots \leftarrow x_n).$ Тогда $P$ является диаграммой Юнга и соответствует перестановке $\pi$. Пусть диаграмма $Q$ получается из диаграммы $P$ пут\"{е}м замены $x_i$ на $i$ для всех $i$ от 1 до $n$. Тогда $Q$ также является диаграммой Юнга.

Далее в работе \cite{Sch61} показывается, что привед\"{е}нное построение пар диаграмм Юнга $(P, Q)$ по перестановкам $\pi\in S_n$ взаимнооднозначно.
\end{proof}

Из алгоритма, привед\"{е}нного в доказательстве леммы \ref{lem:schlem3} следует

\begin{lemma}[\cite{Sch61}]\label{lem:schth1}
Количество строк в диаграмме $P$ равно длине максимальной убывающей подпоследовательности символов в $\pi = x_1 x_2\dots x_n$.
\end{lemma}

Приступим теперь непосредственно к доказательству теоремы \ref{th:main}.

Рассмотрим перестановку $\pi = x_1 x_2\dots x_n$. Она не $(k+1)$-разбиваема тогда и только тогда, когда в соответствующих ей диаграммах $P$ и $Q$ не больше $k$ строк.

Покрасим числа от 1 до $n$ в $k$ цветов произвольным образом. Таких раскрасок $k^n$. Рассмотрим теперь таблицы (не Юнга!), построенные следующим образом. Теперь для каждого $i$ от 1 до $k$ поместим в $i$-ю строку строимой таблицы числа $i$-го цвета в возрастающем порядке так, чтобы наименьшее число в строке стояло в первом столбце и между числами в одной строке не было пустых ячеек (но целиком пустые строки быть могут). Назов\"{е}м полученные таблицы {\it таблицами типа $\beta (n, k)$}. Между раскрасками в $k$ цветов чисел от 1 до $n$ и таблицами типа $\beta(n, k)$ есть естественная биекция, следовательно, таблиц типа $\beta(n, k)$ будет ровно $k^n$. Заметим, что любая диаграмма Юнга, заполненная числами от 1 до $n$ с не более, чем $k$ строками, будет таблицей типа $\beta(n, k)$. Будем считать, что таблицы $A$ и $B$ типа $\beta(n, k)$ эквивалентны ($A\sim_\beta B$), если одну из другой можно получить при помощи перестановки строк. Тогда если в таблице типа $\beta(n, k)$ не больше одной пустой строки, то в соответствующем классе эквивалентности будет ровно $k!$ элементов. Так как в диаграммах Юнга числа в столбцах строго упорядочены по возрастанию, то в каждом классе эквивалентности таблиц типа $\beta(n, k)$ будет не более одной диаграммы Юнга. Если в диаграмме Юнга ровно $k$ строк, то в соответствующей таблице типа $\beta(n, k)$ не будет пустых строк. Следовательно, диаграмм Юнга, заполненных числами от 1 до $n$ и имеющих ровно $k$ строк, не больше, чем ${k^n\over k!}$.

Если в диаграмме Юнга $k$ строк, то в ней не больше, чем $(n-k+1)$ столбец. Раскрасим числа от 1 до $n$ в $(n-k+1)$ цвет. Рассмотрим теперь таблицы (не Юнга!), построенные следующим образом. Теперь для каждого $i$ от 1 до $(n-k+1)$ поместим в $i$-й столбец строимой таблицы числа $i$-го цвета в возрастающем порядке так, чтобы наименьшее число в столбце стояло в первой строке и между числами в одном столбце не было пустых ячеек (но целиком пустые стобцы быть могут). Назов\"{е}м полученные таблицы {\it таблицами типа $\gamma (n, k)$}. Между раскрасками в $(n-k+1)$ цветов чисел от 1 до $n$ и таблицами типа $\gamma (n, k)$ есть естественная биекция, следовательно, таблиц типа $\gamma (n, k)$ будет ровно $k^n$. Заметим, что любая диаграмма Юнга, заполненная числами от 1 до $n$ с $k$ строками, будет таблицей типа $\gamma (n, k)$. Будем считать, что таблицы $A$ и $B$ типа $\gamma (n, k)$ эквивалентны ($A\sim_\gamma B$), если одну из другой можно получить при помощи перестановки столбцов. Пусть в таблице $A$ ровно $t$ ненулевых столбцов. Всего таблиц типа $\gamma (n, k)$ с $t$ ненулевыми строками будет не более, чем таблиц типа $\gamma (n, n-t+1)$, т. е. не более $t^n$. В классе эквивалентности таблицы типа $\gamma (n, k)$ с $t$ непустых столбцов будет $(\min\{t + 1, n-k+1\})!$ элементов. При этом таблиц с $(n-k)$ или $(n-k+1)$ столбцов будет не более $(n-k+1)^n$ и в каждом классе эквивалентности среди них будет $(n-k+1)!$ элементов. Так как в диаграммах Юнга числа в строках строго упорядочены по возрастанию, то в каждом классе эквивалентности таблиц типа $\gamma (n, k)$ будет не более одной диаграммы Юнга. Следовательно, среди таблиц типа $\gamma (n, k)$ будет не более ${(n-k+1)^n\over (n-k+1)!} + \sum\limits_{ t=1}^{n-k-1}{t^n\over t!}\leqslant {(n-k+1)^n\over (n-k)!}$ диаграмм Юнга.

Значит, пар диаграмм Юнга, в каждой из которых по $k$ строк, не больше, чем $\min\{{(n-k+1)^{2n}\over ((n-k)!)^2}, {k^{2n}\over (k!)^2}\}$. Следовательно, существует не больше $\min\{{(n-k+1)^{2n}\over ((n-k)!)^2}, {k^{2n}\over (k!)^2}\}$ перестановок $\pi\in S_n$ с длиной максимальной убывающей подпоследовательности ровно $k$.

Каждая перестановка соответствует с точностью до изоморфизма паре линейных порядков из $n$ элементов. Порядок в перестановочно упорядоченном множестве есть пересечение двух линейных порядков. Так как у каждой пары линейных порядков ровно одно их пересечение, то по леммам \ref{lem:schlem3} и \ref{lem:schth1} количество перестановочно упорядоченных множеств порядка $n$ с максимальной антицепью длины $k$  не больше, чем $\min\{{(n-k+1)^{2n}\over ((n-k)!)^2}, {k^{2n}\over (k!)^2}\}$. Тем самым теорема \ref{th:main2} доказана.

\begin{remark}
Отметим, что по перестановочно упорядоченному множеству не всегда можно определить, какой именно парой линейных порядков оно порождено. Например, рассмотрим множество $\{ p_i \}_{ i=1}^{15}$ с порядком $(p_1>p_2>p_3, p_4>p_5>\dots >p_8, p_9>\dots >p_{15})$. Оно могло быть порождено:
\begin{itemize}
\item парой линейных порядков с соотношениями $(p_3>p_4, p_8>p_9)$ и\linebreak[4] $(p_3<p_4, p_8<p_9)$,
\item парой линейных порядков с соотношениями $(p_3>p_9, p_15>p_1)$ и\linebreak $(p_3<p_9, p_{15}<p_1)$.
\end{itemize}

Эти 2 пары линейных порядков не изоморфны друг другу.
\end{remark}

Оценим $\Delta_k(n)$ --- количество диаграмм Юнга, заполненных числами от 1 до $n$ и имеющих не больше $k$ строк.

\begin{lemma}
Верно неравенство $\Delta_k(n)\leqslant{k^n\over (k-1)!}$.
\end{lemma}
\begin{proof}
Как показывалось ранее, если в таблице типа $\beta(n, k)$ не больше одной пустой строки, то в соответствующем классе эквивалентности будет ровно $k!$ элементов. Следовательно, диаграмм Юнга, заполненных числами от 1 до $n$ и имеющих либо $(k-1)$, либо $k$ строк, не больше ${k^n\over k!}$. Значит, для $k < 3$ лемма доказана. Пусть она доказана для $k<t$. Тогда для $k=t$ имеем $\Delta_k(n)\leqslant {k^n\over k!}+\sum\limits_{i=1}^{ k-2}{i^n\over (i-1)!}\leqslant {k^n\over (k-1)!}.$
\end{proof}

Значит, пар диаграмм Юнга порядка $n$, в каждой из которых по  $\leqslant k$ строк, не больше, чем ${k^{2n}\over ((k-1)!)^2}$. Следовательно, по леммам \ref{lem:schlem3} и \ref{lem:schth1} количество не $(k+1)$-разбиваемых перестановок $\pi\in S_n$ меньше ${k^{2n}\over ((k-1)!)^2}$. Тем самым теорема \ref{th:main} доказана.

Выведем из теоремы \ref{th:main} следствие \ref{cor:main}. Для каждого набора букв $a_{ i_1}, a_{ i_2},\dots , a_{ i_n}$ количество не $(k+1)$-разбиваемых полилинейных слов длины $n$, составленных из этого набора букв, не больше, чем ${k^{2n}\over ((k-1)!)^2}$. Каждому полилинейному слову отвечает ровно один набор из $n$ букв. Так как наборов из $n$ букв ровно $\bigl(\begin{smallmatrix}l\\ n \end{smallmatrix}\bigr)$, то количество не  $(k+1)$-разбиваемых полилинейных слов длины $n$ не больше, чем ${l!k^{2n}\over n!(l-n)!((k-1)!)^2}.$ Тем самым, следствие \ref{cor:main} доказано.

\section{Обобщенные диаграммы Юнга и их производящие функции}

\begin{lemma}[\cite{Kn70}]\label{lem:knuth}
Существует взаимооднозначное соответствие между наборами типа $\alpha(N)$  и парами $(P, Q)$ обобщ\"{е}нных диаграмм Юнга порядка $N$ у которых форма $P$ совпадает с формой $Q$.
\end{lemma}

\begin{proof}
Определим операцию $S\leftarrow x$, где $S$ --- обобщ\"{е}нная диаграмма Юнга, $x$ --- натуральное число, так же, как в доказательстве леммы \ref{lem:schlem3}. Сопоставим некоторому набору типа $\alpha(N)$ из пар $(u_1, v_1), (u_2, v_2),\dots (u_N, v_N)$ диаграмму Юнга $P = (\dots ((v_1\leftarrow v_2)\leftarrow v_3)\dots \leftarrow v_N)$. Пусть диаграмма $Q$ получается из диаграммы $P$ пут\"{е}м замены $v_i$ на $u_i$ для всех $i$ от 1 до $N$. Тогда $Q$ также является диаграммой Юнга.

Далее в работе \cite{Kn70} показывается, что привед\"{е}нное построение пар обобщ\"{е}нных диаграмм Юнга $(P, Q)$ по наборам типа $\alpha$ взаимнооднозначно.
\end{proof}

\begin{notation}
Перестановка $\pi\in S_n$ является набором типа $\alpha(n)$ из пар
 $(1, \pi(1)),\dots ,(n, \pi(n))$.
\end{notation}

{\it Симметрические функции.}

Здесь и далее считаем, что множество индексов при переменных симметрический функций является множеством натуральных чисел.

Напомним несколько понятий из теории симметрических функций.

{\it Полная симметрическая функция} $h_n$ равна $h_n = \sum\limits_{i_1\leqslant i_2\leqslant\dots\leqslant i_n}x_{i_1}x_{i_2}\dots x_{i_n}.$

Пусть $\lambda$ --- набор $(\lambda_1, \lambda_2,\dots ,\lambda_k)$ для некоторого натурального $k$. Пусть также $|\lambda| = \sum\limits_{i=1}^k \lambda_i$. Набор $\lambda$ называется {\it разбиением,} если $\lambda_1\geqslant \lambda_2\geqslant\dots \geqslant\lambda_k.$

{\it Функция Шура} $S_\lambda$ равна $S_\lambda = \det(h_{\lambda_i + j-i})_{1\leqslant i, j\leqslant k}.$

В работе \cite{Ges90} определяются функции $b_i = \sum\limits_{n=0}^{\infty}{x^{2n + i}\over n!(n+i)!}$ и $U_k = \det(b_{\mid i-j\mid})_{1\leqslant i, j\leqslant k}.$

Также вводится функция $R_k(x, y)$ как $R_k(x, y) = \sum\limits_k s_\lambda(x) s_\lambda(y)$, где сумма бер\"{е}тся по всем разбиениям на не более чем $k$ частей. Тогда коэффициент при $x_1 x_2\dots x_n y_1 \dots y_n$ в функции $R_k(x, y)$ равен $\xi_k(n).$ Из этого в \cite{Ges90} выводится, что $U_k = \sum\limits_{n=0}^{\infty}\xi_k(n){ x^{ 2n}\over (n!)^2}.$

Количество полилинейных слов длины $n$ над $l$-буквенным алфавитом $(n \leqslant l)$, в каждом из которых не найд\"{е}тся последовательности из $(k+1)$ буквы в
порядке лексикографического убывания есть $\bigl(\begin{smallmatrix}l\\ n \end{smallmatrix}\bigr)\xi_k(n)$.

\chapter{Дальнейшее улучшение оценок высоты}\label{final}


Представленная вниманию читателя техника, возможно, позволяет улучшить
полученную в данную работе оценку, но при этом она останется только
субэкспоненциальной. Для получения полиномиальной оценки, если она
существует, требуются новые идеи и методы.

В главах \ref{ch:nil} и \ref{ch:height}
подслова большого слова используются прежде всего в качестве
множества независимых элементов, а не набора тесно связанных друг с
другом слов. Далее используется то, что буквы внутри
подслов раскрашены. При уч\"{е}те раскраски только первых букв
подслов получается экспоненциальная оценка. При рассмотрении
раскраски всех букв подслов опять получается экспонента. Данный
факт имеет место из-за построения иерархической системы подслов. Не исключено,
что подробное рассмотрение приведенной связи подслов вкупе с
изложенным выше решением позволит улучшить полученную оценку
вплоть до полиномиальной.

В работе получены оценки на высоту, линейные по числу образующих $l$. На самом деле точные оценки на высоту также линейны по $l$. Следовательно, если какие-либо оценки будут доказаны для случая $l=2$, то с помощью перекодировки образующих можно получить общий случай. Модельный пример применения механизма перекодировки можно найти в секции \ref{s:cod}. Заметим, что в этой секции оценки изначально доказываются не для конкретного числа образующих, а для конкретного базиса Ширшова. 

Представляется перспективным перевод основных понятий доказательства теоремы Ширшова на язык графов. По написанному выше, можно считать, что у нас две образующие: 0 и 1. \index{Граф!подслов} Рассмотрим некоторое очень длинное не $n$-разбиваемое слово $W$ с ракрашенными в соответствии с теоремой Дилуорса позициями (см., например, подсекцию \ref{s:nil:basic_properties}). Теперь возьм\"{е}м подслово $u$ слова $W$, достаточно большое для того, что если мы возьм\"{е}м подслово $v$ слова $W$, в два раза большее $u$ по длине и, в свою очередь, содержащее $u$ как подслово, то число цветов позиций, встречающихся в $v$, примерно равно аналогичному числу в $u$. Рассмотрим теперь бинарное корневое дерево. Отметим у каждой невисячей вершины левого сына как 0, а правого --- как 1. Корневую вершину никак отмечать не будем. Пусть глубина дерева крайне мала по сравнению с длиной слова $u$. Заметим, что для любого натурального $k$ любое слово над бинарным алфавитом длины $k$ может быть представлено как путь длины $k$, начинающийся из корня рассмотренного бинарного графа. 

Теперь для каждого подслова слова $u$ длины $k$ рассмотрим в графе соответствующий путь и покрасим этот путь в цвет первой позиции соответствующего $k$-начала. Естественно, некоторые ребра будут покрашены по несколько раз. Полученная картина --- слишком п\"{е}страя, чтобы сделать какие-либо выводы. Поэтому для каждого цвета оставим самый левый путь этого цвета. Назов\"{е}м полученную структуру {\em деревом подслов}\index{Дерево подслов} (сравните это дерево с наборами $B^p(i)$ из подсекции \ref{ss:set_bpi}). Так как $u$ --- подслово слова $W$, можно представить себе его как ``окно'' определ\"{е}нной длины, положенное на слово $W$. Теперь будем двигать это окно вправо шагом в одну позицию. На каждом шаге будем перерисовывать дерево позиций. Назов\"{е}м изменение дерева позиций при движении окна вправо {\em эволюцией дерева подслов}\index{Эволюция дерева подслов}. Пусть в слове $W$ нет периодов длины $n$, то есть рассматриваем так называемый ниль-случай. Если взять $k=n$, то по лемме \ref{c:lem1.3} при сдвиге окна на $n^2$ позиций дерево позиций точно изменится. Если дерево позиций хорошо сбалансировано, то есть мало групп цветов, имеющих длинную общую часть пути, то дерево довольно быстро эволюционирует, более того, количество изменений в н\"{е}м будет ограничено полиномом. Однако если дерево подслов не сбалансировано, то некоторые ветки дерева ``перегружаются'' цветами. В подсч\"{е}те того, до какой степени ветки могут быть перегружены цветами, быть может, кроется получение полиномиальной оценки на высоту.

Рассмотренный выше граф одинаково применим как для оценки индекса нильпотентности, так и для оценки существенной высоты. Ниже построен граф, который можно построить на периодических подсловах при оценке существенной высоты. Пусть $t$ --- длина периода.

Пусть слово $W$ не $n$-разбиваемо. Как и прежде, рассмотрим некоторое множество попарно непересекающихся несравнимых подслов слова $W$ вида $z^m,$ где $m>2n,$ $z$ --- $t$-буквенное нециклическое слово.  Будем называть
элементы этого множества {\it представителями}, имея в виду, что эти
элементы являются представителями раз\-лич\-ных классов
эквивалентности по сильной сравнимости. Пусть набралось $t$
таких предс\-та\-ви\-те\-лей. Пронумеруем их всех в порядке положения в
слове $W$ (первое --- ближе всех к началу слова) числами от $1$ до $t$. В каждом выбранном представителе в качестве подслов содержатся ровно $t$ различное $t$-буквенное слово.

Введ\"{е}м порядок на этих словах следующим образом: $u\prec v$, если
\begin{itemize}
    \item $u$ лексикографически меньше $v$;
    \item представитель, содержащий $u$ левее представителя, содержащего $v$.
\end{itemize}
Из не сильной
$n$-разбиваемости получаем, что максимальное возможное число
попарно несравнимых элементов равно $t$. По теореме Дилуорса существует разбиение рассматриваемых $t$-буквенных слов на $t$ цепь. Раскрасим слова в $t$ цвет в соответствии с их принадлежностью к цепям. Раскрасим позиции, с которых начинаются слова, в те же цвета, что и соответствующие слова.

 \index{Слово!-цикл}
Напомним, что {\em слово-цикл $u$} --- слово $u$ со всеми его сдвигами по циклу.

Рассмотрим ориентированный граф $G$ с вершинами вида $(k,i)$, где \index{Граф!подслов!большой длины}
$0 < k < n$ и $0 < i \leqslant l$. Первая координата обозначает цвет, а вторая --- букву.

Ребро с некоторым весом $j$
выходит из $(k_1, i_1)$ в $(k_2, i_2)$, если
\begin{itemize}
    \item для некоторых $i_3,i_4,\ldots,i_t$ в $j$-ом представителе содержится слово-цикл\linebreak[3] $i_1i_2\cdots i_t$;
    \item позиции, на которых стоят буквы $i_1, i_2$ раскрашены в цвета $k_1, k_2$ соответственно.
\end{itemize}

Таким образом, граф $G$ состоит из ориентированных циклов длины $t$. Теперь нам требуется найти показатель, который бы строго монотонно рос с появлением каждого нового представителя при движении от начала к концу слова $W.$ Можно заметить, что как и в случае дерева подслов, мы естественным образом столкнулись с понятием эволюции графов. Только в данном случае ``окно'' может ``растягиваться'', то есть его левый край оста\"{е}тся на месте, а правый движется вправо. Разбалансировка же выражается также --- в длинных путях, которые по очереди входят в разные циклы длины $t$. Отметим, что конструкция графа $G$ близка к конструкции графов Рози. 
Обзор тематики графов Рози можно найти в \cite{Lot83}.

Интересно также получить оценки на высоту алгебры над множеством
слов степени не выше сложности алгебры (в англоязычной литературе $\PI$-degree). В работе
\cite{BBL97} получены экспоненциальные оценки, а для
слов, не являющихся линейной комбинацией лексикографически меньших,
в работе \cite{Bel07} получены надэкспоненциальные оценки.

\newpage

\chapter*{Приложение 1. Комбинаторика слов} 

\addcontentsline{toc}{chapter}{\numberline {}Приложение 1. Комбинаторика слов}

\markboth{}{Приложение 1. Комбинаторика слов}

\begin{enumerate}
	\item Проблема Куроша-Левицкого для конечно порожд\"{е}нных \index{Проблема!Куроша}
\begin{itemize}
	\item ниль-алгебр конечного ниль-индекса
	\item алгебр конечного индекса
\end{itemize}
{\bf Теорема Ширшова о высоте.}\index{Теорема!Ширшова о высоте} Множество всех не $n$-разбиваемых слов в конечно порожд\"{е}нной алгебре с допустимым полиномиальным тождеством имеет ограниченную
 высоту $H$ над множеством слов степени не выше $n-1$.

Литература: \cite{Kur41, BBL97, SGS78}.

	\item {\bf Определение.} Ассоциативное слово называется {\em правильным}, если оно лексикографически больше любого своего циклического сдвига. \index{Слово!правильное}

Неассоциативное слово называется {\em правильным}, если оно правильное в ассоциативном смысле и 
\begin{itemize}
	\item если $[u] = [[v][w]]$, то $v$ и $w$ --- правильные слова,
	\item если $[u] = [[v_1] [v_2]] [w]$, то $v_2 \leqslant w$.
\end{itemize}

{\bf Теорема Ширшова}.\index{Теорема!Ширшова для алгебр Ли} В правильном в ассоциативном смысле слове существует единственный способ расставить Лиевы скобки так, чтобы полученное слово было правильным в неассоциативном смысле. 

Правильные слова образуют базис свободной алгебры Ли.

Литература: \cite{BBL97, SGS78, Sh58, Ufn89}. 

	\item {\bf Определение.}\index{Слово!полуправильное} Назов\"{е}м слово $u$ полуправильным, если
любой его конец либо лексикографически меньше $u$, либо является
началом $u$.

{\bf Теорема.} Любое бесконечное слово над конечным алфавитом содержит подслово $fgf$, где $f$ --- полуправильное, а $g$ --- правильное (возможно, пустое) слово.

Литература: \cite{BBL97},  \cite{Ufn89}. 

	\item {\bf Теорема Ван дер Вардена.}\index{Теорема!Ван дер Вардена} Пусть $n$ и $k$ --- натуральные числа, последовательность натуральных чисел разбита на $k$ множеств. Тогда найд\"{е}тся число $f(n, k)$ такое, что среди первых $f(n, k)$ натуральных чисел найд\"{е}тся арифметическая прогрессия длины $n$ из одного множества. 

Многомерное обобщение для фигур и гомотетии с положительным коэффициентом.

Литература: \cite{Bug06, Hin79}. 

	\item Построение фрактала Рози с помощью чисел Трибоначчи.

Литература: \cite{wiki1}. 

	\item {\bf Теорема.}\index{Слово!Зимина}  Слово от $n$ букв избегаемо тогда и только тогда, когда ни одно из его значений не является подсловом $Z_n$. 

Литература: \cite{Ufn89, Sap14, Zim82}. 
	
	\item {\bf Определение.} Группа удовлетворяет условию $C'(\lambda)$, когда общая часть любых двух порождающих соотношений меньше, чем $\lambda$, умноженное на длину любого из них.

{\bf Лемма Гриндлингера.}\index{Лемма!Гриндлингера} В карте, удовлетворяющей условию $C'({1\over 6})$, найд\"{е}тся клетка, большая часть границы которой лежит на границе карты.

Алгебраическая формулировка с группами и соотношениями.

Литература: \cite{Sap14, Kl09}. 

	\item {\bf Теорема Регева.}\index{Теорема!Регева} Если алгебры $A$ и $B$ удовлетворяют полиномиальному тождеству, то алгебра  $A\otimes_F B$ также удовлетворяет полиномиальному тождеству. 

Литература: \cite{Reg71, Lat72}. 

	\item Биекция между словами и парами диаграмм Юнга.

Литература: \cite{Sch61, Kn70}. 

	\item {\bf Определение.}  $p_w(n)$ --- количество подслов длины $n$ в слове $w$.

Слово $w$ называется {\em уравновешенным}, если в любых двух его подсловах одинаковой длины количество единиц различается не больше, чем на 1.\index{Слово!уравновешенное}

 {\bf Теорема.} \index{Слово!Штурма}Эквивалетность тр\"{е}х определений слов Штурма:
\begin{itemize}
	\item $p_w(n) = n+1$
	\item $w$ --- уравновешенное непериодичное слово
	\item $w$ --- механическое слово с иррациональным углом наклона
\end{itemize}
 Литература: \cite{Lot83, Fr11}. 

\end{enumerate}

\newpage

\chapter*{Приложение 2. Алгоритмические методы в теории колец} 

\addcontentsline{toc}{chapter}{\numberline {}Приложение 2. Алгоритмические методы в теории колец}

\markboth{}{Приложение 2. Алгоритмические методы в теории колец}

\begin{enumerate}

	\item Привед\"{е}нные ниже факты отдельно доказывались для $\epsilon$- и $Lie$- алгебр, но доказательства и формулировки для различных типов алгебр похожи, поэтому ниже приведена попытка объединения формулировок.

А.И. Ширшов вв\"{е}л и доказал следующие понятия и теоремы:

{\bf Определение} \index{Слово!правильное}Слова длины 1 назовем {\em $\Omega$-правильными}
($\Omega = K, AK, Lie$) словами и произвольно упорядочим. Считая, что $\Omega$-правильные слова, длина которых меньше $n$, $n >1$, уже определены и упорядочены каким-то способом так, что слова меньшей длины предшествуют
словам большей длины, назовем слово $w$ длины $n$ {\em $\Omega$-правильным}, если
\begin{enumerate}
 \item $w =uv$, где $u$, $v$ --- $\Omega$-правильные слова;
 \item $u\geqslant v$ при $\Omega = K$ и $u>v$ при $\Omega = AK, Lie$;
 \item (Только для $\Omega=Lie$) если $u=u_1u_2$, то $u\leqslant v$.
\end{enumerate}

{\bf Теорема 1.} Правильные слова образуют базис свободной $\Omega$-алгебры.

{\bf Теорема 2.} Всякая подалгебра свободной $\Omega$-алгебры свободна. 

{\bf Проблема тождества для $\Omega$-алгебр.}\index{Проблема!тождества} Существует ли алгоритм, который для произвольного конечного множества $S$ и произвольного элемента $a$ из $\Omega$-алгебры позволяет выяснить, принадлежит ли $a$ идеалу $\langle S\rangle$.

{\bf Теорема о тождестве 1.} Пусть $S$ --- некоторое фиксированное множество элементов свободной $\Omega$-алгебры $E$. Тогда существует алгоритм, позволяющий за конечное число шагов определить, принадлежит ли произвольный элемент $t\in E$ идеалу $\langle S\rangle$.

{\bf Следствие.} Существует алгоритм, решающий проблему тождества для алгебр Ли с одним определяющим соотношением.

{\bf Теорема о тождестве 2.} Существует алгоритм, решающий проблему тождества для алгебр Ли с однородными множествами определяющих соотношений.

{\bf Теорема о свободе.}\index{Теорема!о свободе} Пусть $E_0$ --- $\Omega$-алгебра с множеством порождающих $R$ и одним опредлеяющим соотношением $s=0$, в левую часть которого входит образующий $a_\alpha$. Тогда подалгебра $E'_0$, порожд\"{е}нная в алгебре $E_0$ множеством $R\setminus a_\alpha$, свободна.

Литература: \cite{Sh53, Sh54, Sh58, Sh62(1), Sh62(2), BBL97}. 

	\item {\bf Определение.} Группа удовлетворяет условию $C'(\lambda)$, когда общая часть любых двух порождающих соотношений меньше, чем $\lambda$, умноженное на длину любого из них.

{\bf Лемма Гриндлингера.}\index{Лемма!Гриндлингера} В карте, удовлетворяющей условию $C'({1\over 6})$, найд\"{е}тся клетка, большая часть границы которой лежит на границе карты.

Алгебраическая формулировка с группами и соотношениями.

Алгоритм Дена-Гриндлингера определения тривиальности группового слова в группе с конечным числом соотношений.

Литература: \cite{Sap14, Kl09}. 

	\item {\bf Diamond-lemma.}\index{Даймонд-лемма} Пусть $M$ --- ЧУМ, в котором любая убывающая цепь --- конечна. 

{\bf Определение.} {\em Отношение Ч\"{е}рча-Россера:}\index{Отношение Ч\"{е}рча-Россера} $x\leftrightsquigarrow y$, если у $x$ и $y$ есть общий потомок.
\index{Граф!Ньюмана}
Представим $M$ в виде графа Ньюмана с множеством р\"{е}бер $R$. Тройка $(M, \leqslant, R)$ называется {\em схемой симплификации}. Следующие условия эквивалентны:
\begin{enumerate}\index{Свойство!каноничности}
\item $M$ --- обладает свойством каноничности (т.е. у каждого $m\in M$ нормальная форма единственна).
\item Отношение Ч\"{е}рча-Россера --- транзитивно.
\item Выполняется условие локального слияния (``у любых двух братье есть общий потомок'').
\item В любой компоненте связности лежит ровно один минимальный элемент.
\item ($x\backsim y$, т.е. между $x$ и $y$ есть неориентированный путь) $\Longleftrightarrow (x\leftrightsquigarrow y)$.
\end{enumerate}

\medskip
{\bf Определение.} Введ\"{е}м на мономах $X^*$ линейный порядок $<$ такой, что для любого монома $z\in X^*$ имеет место $x<y\Rightarrow xz<yz$. {\em Базис Гр\"{е}бнера-Ширшова}\index{Базис!Гр\"{е}бнера--Ширшова} некоторого идеала $I\vartriangleleft  k\langle X\rangle$ --- это конечное множество полиномов $G$, порождающее идеал $I$, прич\"{е}м старший моном $\bar{h}$ любого полинома $h\in I$ делится на некоторый старший моном $\bar{g}$ полинома из базиса Гр\"{е}бнера-Ширшова.

Элемент $h$ обладает {\em $H$-представлением}\index{$H$-представление} относительно системы порождающих $G$, если в представлении $h=\sum \alpha_i u_ig_iv_i$ любой моном $u_i\bar{g_i}v_i$ не больше, чем $\bar{h}$.

Определим для полинома $f\in k\langle X\rangle$ его $supp(f)$ --- упорядоченное множество составляющих его мономов. Тогда лексикографический порядок на суппортах полиномов индуцирует частичный порядок $\leqslant_{supp}$ на полиномах $k\langle X\rangle$.

{\bf Теорема.} Следующие условия эквивалентны.

\begin{enumerate}
\item $G$ --- базис Гр\"{е}бнера-Ширшова $I$.
\item Любой элемент $I$ редуцируется относительно $G$ к нулю.
\item Любой $h\in I$ обладает $H$-представлением относительно $G$.
\item Схема  $(k\langle X\rangle, \leqslant_{supp}, R_G)$ обладает свойством каноничности.
\end{enumerate}

Литература: \cite{Ber78, BMM95, Lat12}.

	\item {\bf Теорема (Морс--Туэ, \cite{Th1906, Mor21}).}\index{Теорема!Морса--Туэ} 
Пусть $X = \{a, b\}$, $X^*$ --- множество слов над алфавитом $X$, подстановка $\phi$ задана соотношениями $\phi(a)=ab$, $\phi(b)=ba$. Тогда если слово $w\in X^*$ --- бескубное, то и $\phi(w)$ --- бескубное.

{\bf Теорема (Туэ--1, \cite{Th1906}).}\index{Теорема!Туэ}
 Пусть $X = \{a, b, c\}$, $X^*$ --- множество слов над алфавитом $X$, подстановка $\phi$ задана соотношениями $\phi(a)=abcab$, $\phi(b)=acabcb$, $\phi(c)=acbcacb$. Тогда если слово $w\in X^*$ --- бесквадратное, то и $\phi(w)$ --- бесквадратное.

{\bf Теорема (Туэ--2).}\index{Теорема!Туэ} 
 Пусть $L$ и $N$ --- алфавиты, $N^*$ --- множество слов над алфавитом $N$, для подстановки $\phi:L\rightarrow N^*$ выполнены следующие условия:
\begin{enumerate}
\item если длина $w$ не больше 3, то $\phi(w)$ --- бесквадратное;
\item если $a$, $b$ --- буквы алфавита $L$, а $\phi(a)$ --- подслово $\phi(b)$, то $a=b$. 
\end{enumerate}
Тогда если слово $w\in L^*$ --- бесквадратное, то и $\phi(w)$ --- бесквадратное.

{\bf Теорема (Крошмор, \cite{Cr82}).}\index{Теорема!Крошмора}  Пусть $L$ и $N$ --- алфавиты, $N^*$ --- множество слов над алфавитом $N$, $\phi:L\rightarrow N^*$ --- подстановка, $M$ --- наибольший размер образа буквы алфавита $L$ при подстановке $\phi$, $m$ --- наименьший размер образа буквы $L$ при той же подстановке, $k=max\{3, 1+[(M-3)/m]\}$. Тогда подстановка $\phi$ --- бесквадратная в том и только в том случае, когда для любого бесквадратного слова $w$ длины $\leqslant k$ слово $\phi(w)$ будет бесквадратным.

Литература: \cite{Sap14, Th1906, Mor21, Ber77_1, Ber77_2, Cr82}.

	\item  {\bf Определение.} \index{Алгебра!мономиальная}Алгебра $А$ называется {\em мономиальной}, если в ней есть базис определяющих
соотношений вида $c = 0$, где $c$ --- слово от образующих алгебры.

{\em Конечным автоматом} ({\bf КА}) с алфавитом $X$ \index{Автомат!конечный}
входных символов называется ориентированный граф, в котором выделено два (возможно
пересекающиеся) множества вершин, называемых \index{Вершина графа!начальная}{\em начальными} 
и \index{Вершина графа!конечная}{\em финальными (конечными)}\index{Вершина графа!финальная} и каждое ребро помечено буквой из конечного
алфавита $X$. 
Язык $L$ называется \index{Язык!автоматный}\index{Язык!регулярный}
{\em регулярным} или {\em автоматным}, если существует конечный
автомат, допускающий слова из множества $L$ и только их.

Автомат называется {\em детерминированным}, если \index{Автомат!детерминированный}
\begin{enumerate}
\item начальная вершина ровно одна;
\item из любой его вершины не может выходить более одного ребра, помеченного
одной и той же буквой;
\item нет ребер, помеченных пустой цепочкой.
\end{enumerate}

{\bf Предложение.} для всякого недетерминированного
{\bf КА} существует детерминированный {\bf КА},
допускающий то же самое множество слов.

{\bf Определение.} \index{Алгебра!автоматная}Алгебра $A$ называется {\em автоматной},
если множество ее ненулевых слов от образующих
А является регулярным языком.

{\bf Предложение.} Конечно определенная мономиальная
алгебра является автоматной.

{\bf Определение.} \index{Функция роста алгебры}{\em Функция роста $V_A(n)$ алгебры} $A$ --- это размерность пространства,
порожд\"{е}нного словами длины не выше $n$.

Если следующий предел существует,
то его значение называется\index{Размерность!Гельфанда--Кириллова} {\em размерностью Гельфанда--Кириллова}
алгебры $A$ и обозначается $GK(A)$:
$$GK(A)=\lim\limits_{n\rightarrow\infty}{{\ln(V_A(n))\over \ln(n)}}.$$

Пусть $\Gamma(A)$ --- минимальный детерминированный граф автоматной
алгебры $A$. Назовем вершину графа {\em циклической}\index{Вершина графа!циклическая}, если
существует путь, начинающийся и заканчивающийся в этой
вершине. Назовем вершину {\em дважды циклической}\index{Вершина графа!циклическая!дважды}, если существуют
два различных пути, начинающихся и заканчивающихся
в этой вершине и не проходящих ни через одну другую вершину
дважды.

Пусть граф $\Gamma$ не имеет дважды циклических вершин. Назовем
{\em цепью} подграф графа $\Gamma$, состоящий из последовательности
ребер, в которой конец предыдущего ребра является началом
следующего, и никакая вершина не встречается дважды. Назовем
{\em простым графом} подграф графа $\Gamma$, состоящий из конечного
числа циклов, занумерованных числами $1, 2,\dots, d$, причем пары
соседних циклов с номерами $i, i + 1$ соединены ровно одной
цепью, направленной от $i$-гo к $(i+1)$-му циклу. В первый цикл может
входить одна цепь, и из последнего также может выходить
одна цепь. 

{\bf Теорема (Уфнаровский).}\index{Теорема!Уфнаровского} Пусть $A$ --- автоматная
алгебра, $\Gamma(A)$ --- ее минимальный детерминированный граф.
\begin{enumerate}
\item Если $\Gamma(A)$ имеет вершину, принадлежащую двум различным
циклам, то $A$ имеет экспоненциальную функцию роста.
\item Если $\Gamma(A)$ не имеет дважды циклических вершин,
то $A$ имеет степенную функцию роста. Степень роста\index{Функция роста алгебры}\index{Размерность!Гельфанда--Кириллова}
(размерность Гельфанда-Кириллова) равна количеству циклов
в максимальном простом подграфе, содержащемся в $\Gamma(A)$.
\end{enumerate}

{\bf Теорема.} Пусть граф автоматной мономиальной алгебры
$A$ не имеет вершин, принадлежащих двум циклам. Тогда
$A$ вкладывается в алгебру матриц над полем.

{\bf Следствие.} Пусть $A$ --- автоматная мономиальная
алгебра, $\Gamma(A)$ --- ее минимальный детерминированный граф. Тогда следующие уловия эквивалентны:
\begin{enumerate}
\item $\Gamma(A)$ не имеет дважды циклических вершин;
\item алгебра $A$ имеет степенной рост;
\item алгебра $A$ имеет не экспоненциальный рост;
\item алгебра $A$ представима матрицами над полем;
\item в $A$ выполняется полиномиальное тождество.
\end{enumerate}
Литература: \cite{BBL97, Sap14, Ufn89, wMA, wDS}.
\end{enumerate}

\newpage
\addcontentsline{toc}{chapter}{\numberline {}Предметный указатель}
\input{final.ind}

\newpage

\chapter*{Список литературы}

\addcontentsline{toc}{chapter}{\numberline {}Список литературы}

\pagestyle{headings}

\end{document}